\documentclass[11pt]{amsart}

\usepackage[T1]{fontenc}
\usepackage[utf8]{inputenc}
\usepackage{lmodern}
\usepackage{microtype}
\usepackage[margin=1.02in]{geometry}
\usepackage{amsmath,amssymb,amsthm,mathtools}
\usepackage{aliascnt}
\usepackage{enumitem}
\usepackage{xcolor}
\usepackage[colorlinks=true,linkcolor=blue!55!black,citecolor=blue!55!black,urlcolor=blue!55!black]{hyperref}
\usepackage[nameinlink,noabbrev]{cleveref}
\allowdisplaybreaks

\numberwithin{equation}{section}
\newtheorem{theorem}{Theorem}[section]
\newaliascnt{proposition}{theorem}
\newtheorem{proposition}[proposition]{Proposition}
\aliascntresetthe{proposition}
\newaliascnt{lemma}{proposition}
\newtheorem{lemma}[lemma]{Lemma}
\aliascntresetthe{lemma}
\newaliascnt{corollary}{proposition}

\aliascntresetthe{corollary}
\theoremstyle{remark}
\newaliascnt{remark}{proposition}
\newtheorem{remark}[remark]{Remark}
\aliascntresetthe{remark}

\crefname{theorem}{Theorem}{Theorems}
\crefname{proposition}{Proposition}{Propositions}
\crefname{lemma}{Lemma}{Lemmas}
\crefname{corollary}{Corollary}{Corollaries}
\crefname{remark}{Remark}{Remarks}
\Crefname{theorem}{Theorem}{Theorems}
\Crefname{proposition}{Proposition}{Propositions}
\Crefname{lemma}{Lemma}{Lemmas}
\Crefname{corollary}{Corollary}{Corollaries}
\Crefname{remark}{Remark}{Remarks}

\newcommand{\R}{\mathbb R}
\newcommand{\one}{\mathbf 1}
\newcommand{\wt}{\widetilde}
\newcommand{\norm}[2][]{\lVert #2\rVert_{#1}}
\newcommand{\abs}[1]{\lvert #1\rvert}
\newcommand{\ip}[2]{\langle #1,#2\rangle}
\newcommand{\Dom}{\operatorname{Dom}}
\newcommand{\Var}{\operatorname{Var}}
\newcommand{\calE}{\mathcal E}
\newcommand{\pv}{\operatorname{p.v.}}

\title[Fractional Laplacian on the interval]{Eigenvalues and eigenfunctions 
of the fractional Laplacian on the interval}
\author{Cheng Zhang}
\address{Yau Mathematical Sciences Center, Tsinghua University, Beijing, 100084, P.R. China}
\email{czhang98@tsinghua.edu.cn}
\date{}

\begin{document}
\begin{abstract}
We prove a three-term asymptotic formula for the  eigenvalues of the  fractional
Laplacian  on the bounded interval $(-1,1)$. This improves  the eigenvalue asymptotics of Kulczycki--Kwa\'snicki--Ma{\l}ecki--St\'os \cite{KulczyckiKwasnickiMaleckiStos} and Kwa\'snicki \cite{KwasnickiInterval},  and   confirms the conjectural $O_\alpha(n^{-2})$ remainder suggested by the numerical simulations of Kaleta--Kwa\'snicki--Ma{\l}ecki \cite{KaletaKwasnickiMalecki}. Moreover, we prove  that the normalized
eigenfunctions are  bounded uniformly in the eigenvalue index $n$ and the fractional order $\alpha$. This settles the conjecture proposed by Kwa\'snicki \cite{KwasnickiInterval} through numerical experiments. Furthermore, we prove that the $n$-th eigenfunction has exactly $n-1$ zeros in the interval $(-1,1)$ and every zero is simple, and hence there are exactly $n$ nodal domains. A key ingredient in the proof is an explicit representation of the eigenfunction.
\end{abstract}
\maketitle

\section{Introduction}
\label{sec:intro}

Let $I=(-1,1)$ and $0<\alpha<2$.  Let
\begin{equation}\label{eq:singular-integral}
 \mathcal A=(-\Delta)^{\alpha/2},
 \qquad
 \mathcal Af(x)=c_\alpha\,\pv\!\int_\R
 \frac{f(x)-f(y)}{\abs{x-y}^{1+\alpha}}\,dy,
 \qquad
 c_\alpha=\frac{\Gamma(1+\alpha)\sin(\pi\alpha/2)}{\pi}.
\end{equation}
For $\widehat f(\xi)=\int_\R e^{-ix\xi}f(x)\,dx$, its symbol is
$\abs\xi^\alpha$.  The equivalence of the Fourier, singular-integral, and
quadratic-form definitions is \cite{KwasnickiDefinitions}.
We write $\mathcal A_I$ for the zero-exterior realization of $\mathcal A$ on $I$.  For $f\in C_c^\infty(I)$, $\mathcal A_If$ is the restriction of $\mathcal Af$ to $I$.  Its Friedrichs extension, denoted by the same symbol, is self-adjoint on $L^2(I)$ and has discrete spectrum.  Every eigenvalue is simple by \cite{FallEtAl,kwa23,yzm25}. Every eigenfunction $\varphi_n\in C^\infty(I)$ by \cite{RosOtonSerraRegularity,BorthagarayDelPezzoMartinez, GrubbSpectral}.
Thus
\begin{equation}\label{eq:eigensystem}
	\mathcal A_I\varphi_n=\lambda_n\varphi_n,
	\qquad \norm[L^2(I)]{\varphi_n}=1,
	\qquad 0<\lambda_1<\lambda_2<\cdots.
\end{equation}
 Throughout, $\lambda_n$ and $\varphi_n$ denote the $n$th eigenvalue and the $n$th real
$L^2(I)$-normalized eigenfunction corresponding to the order
$\alpha$.  Their dependence on $\alpha$ is suppressed in the notation. It is important to note that $\mathcal A_I$ does not
coincide with the fractional power of the Dirichlet Laplacian on $I$.

The spectral problem studied in this paper has a long history. It is of great interest in physics \cite{ZRK}, and there is a considerable amount of related, mostly numerical,
research in the physics literature. For general background, we refer to the surveys on fractional Sobolev spaces
and definitions of the fractional Laplacian
\cite{DNPV,KwasnickiDefinitions,LischkeEtAl}, and to
the survey \cite{FrankSurvey} of eigenvalue bounds and asymptotics. By translation and scaling, all results on $I=(-1,1)$ extend to any bounded open interval, with the eigenvalues and  eigenfunctions rescaled accordingly.

The one-term Weyl law for $\lambda_n$ 
goes back to Blumenthal--Getoor
\cite{BlumenthalGetoor}.  
The two-sided estimate 
$$
 \frac12\left(\frac{n\pi}{2}\right)^\alpha
 \le \lambda_n\le
 \left(\frac{n\pi}{2}\right)^\alpha$$ is due to DeBlassie \cite{De} and Chen--Song \cite{ChenSongEigen}.  
  Kwa\'snicki \cite[Theorem~1]{KwasnickiInterval} proved the asymptotic formula for every $0<\alpha<2$,
\begin{equation}\label{eq:Kwasnicki-two-term}
 \lambda_n=
 \left(\frac{n\pi}{2}-\frac{(2-\alpha)\pi}{8}\right)^\alpha
 +O_\alpha(n^{-1}).
\end{equation}
The important case $\alpha=1$ is due to Kulczycki--Kwa\'snicki--Ma{\l}ecki--St\'os \cite[Theorem~6]{KulczyckiKwasnickiMaleckiStos}. It
was later extended to  more general operators $\psi(-\Delta)$ by
Kaleta--Kwa\'snicki--Ma{\l}ecki
\cite{KaletaKwasnickiMalecki}.  For the fractional case, numerical simulations \cite[Remark~1.5]{KaletaKwasnickiMalecki} predicted an $O_\alpha(n^{-2})$ remainder in \eqref{eq:Kwasnicki-two-term}.  The forthcoming  \Cref{thm1} establishes a three-term asymptotic formula, thereby confirming this prediction.

The eigenfunction bounds have been extensively studied. See e.g.  \cite{sogge88,smithsogge,HuangSireZhang,jmpa, DeBlassieMendez,ChenWethLogarithmic,FeulefackJarohsWeth}.  
Kulczycki--Kwa\'snicki--Ma{\l}ecki--St\'os \cite[Corollary~5]{KulczyckiKwasnickiMaleckiStos} first obtained the uniform eigenfunction bound for $\alpha=1$
\begin{equation}
	\norm[L^\infty(I)]{\varphi_n}\le 3,\  \ n\ge1.
\end{equation}
Kwa\'snicki  \cite[Proposition~2]{KwasnickiInterval} generalized it  for $\alpha\ge\tfrac12$
 and his numerical
experiments suggest that the same uniform boundedness remains valid
for $0<\alpha<\tfrac12$.  The forthcoming \Cref{thm2} settles this conjecture with a constant uniform simultaneously in \(\alpha\) and \(n\).

The nodal set of $\varphi_n$  is known for some small $n$. For example, the first eigenfunction $\varphi_1$ has no zero in $I$, see e.g. 
\cite{BanuelosKulczyckiMendezHernandez,ServadeiValdinociWeakViscosity,DelPezzoBonderLopez}.  The second eigenfunction $\varphi_2$ is antisymmetric and has exactly one zero in $I$, see \cite{BanuelosKulczyckiSteklov, KwasnickiInterval, DydaKuznetsovKwasnicki, FallFeulefackTemgouaWeth}. The third eigenfunction $\varphi_3$ has exactly 2 zeros in $I$, see \cite{BanuelosKulczyckiSteklov, FallWeth}. For  $\alpha=1$ and every $n\ge1$,
Ba\~nuelos--Kulczycki
\cite[Proposition~5.5]{BanuelosKulczyckiSteklov}
proved that $\varphi_n$ has at most $2n-2$ zeros in $I$. For every $0<\alpha<2$ and  $n\ge1$, Yang--Zhang--Ma \cite{yzm25} proved that $\varphi_n$ changes its signs at most $2n-2$ times in $I$.  The forthcoming   \Cref{thm:nodal} shows that for every $0<\alpha<2$ and $n\ge1$, $\varphi_n$ has exactly $n-1$ zeros in $I$, and every zero is simple, and hence there are exactly $n$ nodal domains.

\subsection{Main results}
Let
\begin{equation}\label{eq:parameters}
 \theta=\frac{(2-\alpha)\pi}{8},
 \qquad
 \mu_n=\frac{n\pi}{2}-\theta,
 \qquad
 \kappa_\alpha=\frac{\alpha c_\alpha}{2^{\alpha+1}}.
\end{equation}
For $0<\alpha<1$, define
\begin{equation}\label{eq:A-alpha}
 A_\alpha=-\frac{2\kappa_\alpha^2}{\alpha(1+2\alpha)},\qquad  B_\alpha=
 \frac{\alpha\Gamma(1+2\alpha)\cos(\pi\alpha)}
 {\pi2^{1+2\alpha}}\tan\frac{\pi\alpha}2.
\end{equation}
For $1<\alpha<2$, define
\begin{equation}\label{eq:eta}
 \eta_\alpha=(1+\alpha)\kappa_\alpha\cot\frac{\pi}{\alpha}.
\end{equation}
Since $\mu_n\approx n$, the asymptotic formula \eqref{eq:Kwasnicki-two-term} is equivalent to
$\lambda_n=\mu_n^\alpha+O_\alpha(\mu_n^{-1})$. 
\begin{theorem}\label{thm1}
For every $0<\alpha<2$, as $n\to\infty$,
\[
 \lambda_n=\mu_n^\alpha
 +(-1)^n\kappa_\alpha\mu_n^{-2}+\mathcal R_n(\alpha).
\]
The remainder has three regimes.
\begin{enumerate}[label=\textup{(\roman*)},leftmargin=2.2em]
\item If $0<\alpha<1$, then
\begin{equation}\label{eq:main-small}
 \mathcal R_n(\alpha)
 =\bigl(A_\alpha+(-1)^nB_\alpha\bigr)\mu_n^{-2-\alpha}
 +o_\alpha(\mu_n^{-2-\alpha}).
\end{equation}
\item If $\alpha=1$, then
\begin{equation}\label{eq:main-critical}
 \mathcal R_n(1)
 =\frac{(-1)^{n+1}}{2\pi^2}\mu_n^{-3}\log\mu_n
 +O(\mu_n^{-3}).
\end{equation}
\item If $1<\alpha<2$, then
\begin{equation}\label{eq:main-large}
 \mathcal R_n(\alpha)
 =(-1)^n\eta_\alpha\mu_n^{-3}+o_\alpha(\mu_n^{-3}).
\end{equation}
\end{enumerate}
\end{theorem}

\begin{theorem}\label{thm2}
Every normalized eigenfunction in \eqref{eq:eigensystem} satisfies
\begin{equation}\label{eq:main-Linfty}
 \norm[L^\infty(I)]{\varphi_n}\le2,
 \qquad 0<\alpha<2,\quad n\ge1.
\end{equation}
\end{theorem}

\begin{theorem}\label{thm:nodal}
For every $0<\alpha<2$ and $n\ge1$, the size of the nodal set of $\varphi_n$ is 
\begin{equation}
	\#\{x\in I:\varphi_n(x)=0\}=n-1,
\end{equation}
and every zero is simple, and hence there are exactly $n$ nodal domains.
\end{theorem}

\subsection{Organization.}
\Cref{sec:halfline} records the half-line estimates used throughout the paper.  \Cref{sec:glued} constructs the quasimode $\wt\phi_n$ and proves the two-term asymptotic formula.  \Cref{sec:large-unified,sec:small} determine the third term in the asymptotic formula for $1\le\alpha<2$ and $0<\alpha<1$, respectively.  \Cref{sec:nodal} derives an exact eigenfunction representation, proves the uniform bound in \Cref{thm2}, and then proves the zero count and simplicity in \Cref{thm:nodal}.
\subsection{Notation}
Throughout, $C$ may change from
line to line, and subscripts indicate its allowed dependence.  An absolute
constant is independent of all varying parameters.  The notation $O_\alpha$
and $o_\alpha$ allows dependence on fixed $\alpha$. Whenever a function defined on a bounded interval is viewed as a function on $\R$, it
is understood to be extended by zero outside that interval, and we retain the
same notation for the extension.  For every bounded interval $J\subset\R$, the form domain is
\[
\wt H^{\alpha/2}(J)
=\{f\in H^{\alpha/2}(\R):f=0\ \text{a.e. on }\R\setminus J\}
\]
and we identify each such function with its restriction to $J$.  The associated quadratic form is
\begin{equation}\label{eq:form}
	\calE_{\alpha,J}(f,g)=\frac{c_\alpha}{2}\iint_{\R^2}
	\frac{(f(x)-f(y))(g(x)-g(y))}
	{\abs{x-y}^{1+\alpha}}\,dx\,dy
	=\frac1{2\pi}\int_\R\abs\xi^\alpha
	\widehat f(\xi)\overline{\widehat g(\xi)}\,d\xi.
\end{equation}
The change
of variables $y=x+h$ gives 
\begin{equation}\label{eq:translation-form}
	\calE_{\alpha,J}[f]:=\calE_{\alpha,J}(f,f)=\frac{c_\alpha}{2}\int_\R
	\frac{\norm[L^2(\R)]{f(\,\cdot+h)-f}^2}{\abs h^{1+\alpha}}\,dh
	=c_\alpha\int_0^\infty
	\frac{\norm[L^2(\R)]{f(\,\cdot+h)-f}^2}{h^{1+\alpha}}\,dh.
\end{equation}

\subsection{Proof outline.}
For \Cref{thm1}, we glue the two generalized half-line eigenfunctions to form the quasimode $\wt\phi_n$ in \eqref{eq:qmode-decomp} as in  \cite{KulczyckiKwasnickiMaleckiStos,KwasnickiInterval,KaletaKwasnickiMalecki}.  The Rayleigh quotient
$\calE_{\alpha,I}[\wt\phi_n]/\|\wt\phi_n\|_{L^2(I)}^2$ gives the
two-term formula in \Cref{thm:two-term}.  To obtain the third term in the asymptotic formula, at $\alpha=1$ we expand the residual
$z_n=(\mathcal A_I-\mu_n^\alpha)\wt\phi_n$ through order
$\mu_n^{-2}\log\mu_n$, and for $1<\alpha<2$ through order
$\mu_n^{-2}$.  For $0<\alpha<1$, the approximation error in
\eqref{eq:ordinary-transfer} has the same order as the desired third term, so
we need to use $w_n$ in \eqref{eq:w} instead of $\wt\phi_n$ to obtain the sharper approximation error in
\eqref{eq:filtered}.

For \Cref{thm2,thm:nodal}, we show that the eigenfunction $\varphi_n$ has a strictly positive bilateral Laplace transform on $(0,\infty)$, and  establish an explicit representation for the eigenfunction  in \Cref{prop:nodal-representation}. Then we obtain the desired results on  eigenfunction bounds and nodal sets by a comparison with  the standard Laplacian eigenfunctions on $I$.

\subsection{Further discussion}
It would be natural to ask how far the present one-dimensional results extend
to the fractional Laplacian in a bounded domain
$\Omega\subset\mathbb R^d$, $d\ge2$.  The  one-term Weyl law is known in broad
generality \cite{BlumenthalGetoor,GeisingerWeyl,GrubbSpectral,GrubbWeyl}.  The two-term
small-time heat-trace expansion, whose second term is proportional to
$|\partial\Omega|$, was proved for smooth and Lipschitz domains in
\cite{BanuelosKulczycki,BanuelosKulczyckiSiudeja}.  A
two-term semiclassical asymptotic formula for eigenvalue sums was proved in \cite{FrankGeisinger}. Ivrii
\cite{IvriiFractional} obtained the two-term Weyl law for the unsmoothed
eigenvalue counting function under the standard non-periodicity condition. 
Important open problems are to remove this condition,  and derive precise asymptotics for
individual eigenvalues in higher-dimensional domains. 
For eigenfunctions, Huang--Sire--Zhang \cite{HuangSireZhang} proved the
Sogge-type interior $L^p$ bounds.  Sharp global $L^p$ and $L^\infty$ estimates
up to the boundary, together with their optimal dependence on $\alpha$, remain
to be determined. Furthermore, it would be interesting to investigate these spectral problems for more general operators $\psi(-\Delta)$ and Schr\"odinger operators with potentials.

\subsection{Acknowledgments}
The author would like to thank Christopher  Sogge, Yannick Sire and Xiaoqi Huang for early discussions. The author is supported in part by the National Key R\&D Program of China
2024YFA1015300 and NSFC Grant 12371097.
\subsection{Statements and Declarations}
Data sharing not applicable to this article as no datasets were generated or analyzed during the current study. The author has no relevant financial or non-financial interests to disclose.
\section{Generalized half-line eigenfunctions}
\label{sec:halfline}

This section records the half-line formulas used later.  We derive the tail
expansion of the correction $G_\alpha$, compute its first moment when it is finite,
and prove the norm identity used to normalize $\wt\phi_n$.

Set
\[
 K_\alpha=c_\alpha\sqrt{\frac{\alpha}{2}},
 \qquad
 M_\alpha=\cos\theta-\sqrt{\frac{\alpha}{2}}.
\]
By \cite[Theorem~1.1 and Example~6.1]{KwasnickiHalfLine}, the generalized
half-line eigenfunction at frequency one is
\begin{equation}\label{eq:F}
 F_\alpha(t)=\sin(t+\theta)-G_\alpha(t),\qquad t>0,
\end{equation}
where $G_\alpha$ is the nonnegative, completely monotone correction term.  We extend
$F_\alpha$ by zero to $(-\infty,0]$.  Then, for every $\tau>0$,
\begin{equation}\label{half}
 \mathcal A F_\alpha(\tau\,\cdot)=\tau^\alpha F_\alpha(\tau\,\cdot)
 \quad\text{on }(0,\infty).
\end{equation}
The endpoint asymptotic in
\cite[Example~6.1 and Lemma~4.27]{KwasnickiHalfLine} is
\begin{equation}\label{eq:F-zero}
 F_\alpha(t)\sim
 \frac{t^{\alpha/2}}
 {\sqrt{\alpha/2}\,\Gamma(\alpha/2)},
 \qquad t\downarrow0.
\end{equation}
Consequently $G_\alpha(0+)=\sin\theta$.  Moreover,
\cite[Lemma~4.21]{KwasnickiHalfLine} gives
\begin{equation}\label{eq:Gmass}
 \int_0^\infty G_\alpha(t)\,dt=M_\alpha.
\end{equation}

Example~6.1 of \cite{KwasnickiHalfLine} also gives
\begin{equation}\label{eq:gamma}
 G_\alpha(t)=\int_0^\infty e^{-ts}\gamma_\alpha(s)\,ds,
 \qquad
 \gamma_\alpha(s)=\frac{K_\alpha}{\Gamma(1+\alpha)}
 \frac{s^\alpha e^{I_\alpha(s)}}
 {1+s^{2\alpha}-2s^\alpha\cos(\pi\alpha/2)},
\end{equation}
where
\begin{equation}\label{eq:I}
 I_\alpha(s)=\frac{s}{\pi}\int_0^\infty
 \frac{\log\!\bigl((1-u^\alpha)/(1-u^2)\bigr)}{s^2+u^2}\,du,
\end{equation}
with the continuous value understood at $u=1$.

\begin{lemma}\label{lem:halfline-data}
As $t\to\infty$, the following expansions hold and remain valid after any
fixed number of differentiations.

If $0<\alpha<1$, then
\begin{equation}\label{eq:Gsmall}
 G_\alpha(t)=K_\alpha t^{-1-\alpha}
 +K_{2,\alpha}t^{-1-2\alpha}+o_\alpha(t^{-1-2\alpha}),
\end{equation}
where
\begin{equation}\label{eq:K2small}
 K_{2,\alpha}=K_\alpha\frac{\Gamma(1+2\alpha)}{\Gamma(1+\alpha)}
 \left(2\cos\frac{\pi\alpha}{2}
 -\frac12\sec\frac{\pi\alpha}{2}\right).
\end{equation}
For $\alpha=1$, $K_1=1/(\pi\sqrt2)$ and
\begin{equation}\label{eq:Gcritical}
 G_1(t)=K_1t^{-2}-\frac{2K_1}{\pi}t^{-3}\log t+O(t^{-3}),
 \qquad
 \int_0^TtG_1(t)\,dt=K_1\log T+O(1).
\end{equation}
If $1<\alpha<2$, then
\begin{equation}\label{eq:Glarge}
 G_\alpha(t)=K_\alpha t^{-1-\alpha}
 +K_{2,\alpha}t^{-2-\alpha}+o_\alpha(t^{-2-\alpha}),
 \qquad
 K_{2,\alpha}=(1+\alpha)K_\alpha\cot\frac{\pi}{\alpha},
\end{equation}
and
\begin{equation}\label{eq:M1}
 M_{1,\alpha}:=\int_0^\infty tG_\alpha(t)\,dt
 =-\sin\theta-\sqrt{\alpha/2}\cot\frac{\pi}{\alpha}.
\end{equation}
Finally,
\begin{equation}\label{eq:halfline-norm-identity}
 \int_0^\infty
 \bigl(F_\alpha(t)^2-\sin^2(t+\theta)\bigr)\,dt
 =-\frac14\sin(2\theta).
\end{equation}
\end{lemma}

\begin{proof}
We first determine $I_\alpha(s)$ as $s\downarrow0$.  For $0<\alpha<1$,
the logarithm in \eqref{eq:I} equals $-u^\alpha+o(u^\alpha)$ at zero and is
bounded there by $C_\alpha u^\alpha$.  Split the integral at a fixed
$\delta\in(0,1)$.  The part over $[\delta,\infty)$ is $O_\alpha(s)$, and the
substitution $u=sv$ in the remaining part gives
\[
 \frac{I_\alpha(s)}{s^\alpha}\longrightarrow
 -\frac1\pi\int_0^\infty\frac{v^\alpha}{1+v^2}\,dv
 =-\frac1{2\cos(\pi\alpha/2)}.
\]
For $\alpha=1$, the same substitution gives
\[
 I_1(s)=-\frac1\pi\int_0^\infty
 \frac{\log(1+sv)}{1+v^2}\,dv.
\]
Differentiating under the integral sign yields
\[
 I_1'(s)=\frac{\log s-(\pi/2)s}{\pi(1+s^2)}.
\]
Since $I_1(s)\to0$, integration gives
$I_1(s)=\pi^{-1}s\log s-\pi^{-1}s+O(s^2)$.

For $1<\alpha<2$, the logarithm in \eqref{eq:I}, divided by $u^2$, is
integrable.  Dominated convergence followed by integration by parts gives
\begin{align*}
 \lim_{s\downarrow0}\frac{\pi I_\alpha(s)}s
 &=\int_0^\infty
 \frac{\log\!\bigl((1-u^\alpha)/(1-u^2)\bigr)}{u^2}\,du\\
 &=\pv\int_0^\infty
 \left(\frac{2}{1-u^2}
 -\frac{\alpha u^{\alpha-2}}{1-u^\alpha}\right)du\\
 &=\pv\int_0^\infty\frac{t^{-1/2}-t^{-1/\alpha}}{1-t}\,dt
 =\pi\cot\frac{\pi}{\alpha}.
\end{align*}
The principal-value identity in the last line follows from
$\pv\int_0^\infty t^{p-1}(1-t)^{-1}\,dt=\pi\cot(\pi p)$, which is a
consequence of the Euler reflection formula.  We have proved
\[
 I_\alpha(s)=
 \begin{cases}
 -\dfrac{s^\alpha}{2\cos(\pi\alpha/2)}+o(s^\alpha),&0<\alpha<1,\\[2mm]
 \dfrac{s}{\pi}\log s-\dfrac{s}{\pi}+O(s^2),&\alpha=1,\\[2mm]
 s\cot(\pi/\alpha)+o(s),&1<\alpha<2.
 \end{cases}
\]

Substitution in \eqref{eq:gamma} now gives
\[
 \gamma_\alpha(s)=
 \begin{cases}
 \dfrac{K_\alpha}{\Gamma(1+\alpha)}s^\alpha
 +\dfrac{K_{2,\alpha}}{\Gamma(1+2\alpha)}s^{2\alpha}
 +o(s^{2\alpha}),&0<\alpha<1,\\[2mm]
 K_1\bigl(s+\pi^{-1}s^2\log s+O(s^2)\bigr),&\alpha=1,\\[2mm]
 \dfrac{K_\alpha}{\Gamma(1+\alpha)}s^\alpha
 +\dfrac{K_{2,\alpha}}{\Gamma(2+\alpha)}s^{1+\alpha}
 +o(s^{1+\alpha}),&1<\alpha<2.
 \end{cases}
\]
Using
\[
 \int_0^\infty e^{-ts}s^q\,ds=\Gamma(q+1)t^{-q-1},
 \qquad
 \int_0^\infty e^{-ts}s^2\log s\,ds
 =-2t^{-3}\log t+O(t^{-3}),
\]
we obtain \eqref{eq:Gsmall}--\eqref{eq:Glarge}.  To differentiate these
expansions, choose a smooth cutoff equal to one on $[0,1]$ and zero on
$[2,\infty)$.  Subtract the displayed terms from $\gamma_\alpha(s)$,
multiply the remainder by this cutoff, and denote the result by $r(s)$.  If $r(s)=o(s^q)$, then, for every fixed $m\ge0$,
\[
 \int_0^\infty e^{-ts}s^m r(s)\,ds=o(t^{-q-m-1})
\]
by the substitution $u=ts$ and dominated convergence.  On the complementary
region $s\ge1$ and for $t\ge1$,
\[
 s^m e^{-ts}\le C_m e^{-t/2}e^{-s/4}.
\]
Since $\int_0^\infty e^{-s/4}\gamma_\alpha(s)\,ds=G_\alpha(1/4)<\infty$,
this part is $O_m(e^{-t/2})$.  Replacing $o(s^q)$ by $O(s^q)$ gives the
corresponding $O$-estimate.  Finally,
$tG_1(t)=K_1t^{-1}+O(t^{-2}\log t)$, and integration from $1$ to $T$ proves
the second formula in \eqref{eq:Gcritical}.

The Laplace-transform formula
\cite[Theorem~1.1, equation~(1.1)]{KwasnickiHalfLine} is
\begin{equation}\label{eq:F-Laplace}
 \int_0^\infty e^{-st}F_\alpha(t)\,dt
 =\frac{\sqrt{\alpha/2}e^{-I_\alpha(s)}}{1+s^2}.
\end{equation}
For $1<\alpha<2$, differentiation at $s=0$ gives
\[
 -\sqrt{\alpha/2}\cot\frac{\pi}{\alpha}
 =\sin\theta+\int_0^\infty tG_\alpha(t)\,dt,
\]
where we used \eqref{eq:F} and the integrability of $tG_\alpha(t)$.  This is
\eqref{eq:M1}.

It remains to prove \eqref{eq:halfline-norm-identity}.  Since
\[
F_\alpha(t)=\sin(t+\theta)-G_\alpha(t),
\qquad
G_\alpha\in L^1(0,\infty)\cap L^\infty(0,\infty),
\]
we have
\[
\begin{aligned}
	&F_\alpha(\lambda x)F_\alpha(\nu x)
	-\sin(\lambda x+\theta)\sin(\nu x+\theta)\\
	&\quad=
	-G_\alpha(\lambda x)\sin(\nu x+\theta)
	-G_\alpha(\nu x)\sin(\lambda x+\theta)
	+G_\alpha(\lambda x)G_\alpha(\nu x).
\end{aligned}
\]
Hence
\[
\int_0^\infty
\bigl(
F_\alpha(\lambda x)F_\alpha(\nu x)
-
\sin(\lambda x+\theta)\sin(\nu x+\theta)
\bigr)\,dx
\]
converges absolutely and is continuous for \(\lambda,\nu>0\).

By \cite[Theorem~1.3]{KwasnickiHalfLine}, the transform
\[
(\mathcal{U}f)(\lambda)=
\sqrt{\frac2\pi}
\int_0^\infty f(x)F_\alpha(\lambda x)\,dx
\]
is unitary on \(L^2(0,\infty)\).  Since $\mathcal{U}\mathcal{U}^*$ is the identity map, we have
\[
\int_0^\infty
F_\alpha(\lambda x)F_\alpha(\nu x)\,dx
=
\frac{\pi}{2}\delta(\lambda-\nu)
\qquad\text{in }\mathcal D'((0,\infty)^2).
\]
Using
\[
\int_0^\infty e^{isx}\,dx
=
\pi\delta(s)+i\,\pv\frac1s
\qquad\text{in }\mathcal S'(\mathbb R),
\]
together with \(\lambda+\nu>0\), we obtain
\[
\int_0^\infty
\sin(\lambda x+\theta)\sin(\nu x+\theta)\,dx
=
\frac{\pi}{2}\delta(\lambda-\nu)
+
\frac{\sin(2\theta)}{2(\lambda+\nu)}
\]
in \(\mathcal D'((0,\infty)^2)\).  Subtracting the last two
distributional identities gives
\[
\begin{aligned}
	\int_0^\infty
	\bigl(
	F_\alpha(\lambda x)F_\alpha(\nu x)
	-
	\sin(\lambda x+\theta)\sin(\nu x+\theta)
	\bigr)\,dx=
	-\frac{\sin(2\theta)}{2(\lambda+\nu)}.
\end{aligned}
\]
Both sides are continuous on \((0,\infty)^2\), so this equality holds
pointwise.  Taking \(\lambda=\nu=1\), we prove \eqref{eq:halfline-norm-identity}.

\end{proof}

\section{Approximation of the interval eigenfunctions}
\label{sec:glued}

This section constructs the quasimode $\wt\phi_n$ to approximate the corresponding eigenfunction.  We use its Rayleigh quotient to prove
the two-term asymptotic formula of the eigenvalue.

Choose $q\in C^\infty(\R)$ such that
\[
 0\le q\le1,\qquad q=0\ \text{on }(-\infty,-1/3],\qquad
 q=1\ \text{on }[1/3,\infty),\qquad q(x)+q(-x)=1.
\]
An explicit $q$ is used in \cite{KwasnickiInterval} and \cite{KaletaKwasnickiMalecki}, but their proofs
 require only the displayed
support, symmetry, and smoothness properties of $q$, so they apply to the
present choice.  Define the quasimode
\begin{equation}\label{eq:qmode}
 \wt\phi_n(x)=q(-x)F_\alpha(\mu_n(1+x))
 +(-1)^{n+1}q(x)F_\alpha(\mu_n(1-x)).
\end{equation}

Since $2\mu_n+2\theta=n\pi$,
\begin{equation}\label{eq:match}
 (-1)^{n+1}\sin(\mu_n(1-x)+\theta)
 =\sin(\mu_n(1+x)+\theta).
\end{equation}
Hence, on $I$,
\begin{equation}\label{eq:qmode-decomp}
 \wt\phi_n(x)=S_n(x)-q(-x)G_\alpha(\mu_n(1+x))
 -(-1)^{n+1}q(x)G_\alpha(\mu_n(1-x)),
 \qquad S_n(x)=\sin(\mu_n(1+x)+\theta).
\end{equation}
Define on $I$
\begin{equation}\label{eq:R-alpha-n}
 R_{\alpha,n}(x)=-K_\alpha\left[
 q(x)(1+x)^{-1-\alpha}
 +(-1)^{n+1}q(-x)(1-x)^{-1-\alpha}\right].
\end{equation}
Then
\begin{equation}\label{eq:R-reflection}
 R_{\alpha,n}(-x)=(-1)^{n+1}R_{\alpha,n}(x).
\end{equation}
Fix also $\chi\in C_c^\infty((1/3,3/2))$ such that $\chi=1$ on $[1/2,1]$.

\begin{lemma}\label{lem:operator-bound}
If $v\in W^{3,1}(\R)$, then $\mathcal Av\in W^{1,1}(\R)$ and
\[
 \norm[W^{1,1}(\R)]{\mathcal Av}
 \le C_\alpha\norm[W^{3,1}(\R)]{v}.
\]
If $d>0$, $v\in L^1(\R)$, and $\operatorname{supp}v\subset[d,\infty)$, then,
for $x\in(-1,0)$,
\[
 \mathcal Av(x)=-c_\alpha\int_d^\infty\frac{v(y)}{(y-x)^{1+\alpha}}\,dy,
 \qquad
 \norm[W^{1,1}(-1,0)]{\mathcal Av}
 \le C_{\alpha,d}\norm[L^1]{v}.
\]
\end{lemma}

\begin{proof}
For $w\in W^{2,1}(\R)$ and $t>0$,
\[
 2w(x)-w(x+t)-w(x-t)
 =-\int_{-t}^t(t-|s|)w''(x+s)\,ds.
\]
This identity and translation invariance give
\[
 \norm[L_x^1]{2w(x)-w(x+t)-w(x-t)}
 \le \min\{t^2\norm[L^1]{w''},4\norm[L^1]{w}\}.
\]
Using
\[
 \mathcal Aw(x)=c_\alpha\int_0^\infty
 \frac{2w(x)-w(x+t)-w(x-t)}{t^{1+\alpha}}\,dt
\]
and splitting at $t=1$ gives
\[
 \norm[L^1]{\mathcal Aw}
 \le c_\alpha\left(
 \frac{\norm[L^1]{w''}}{2-\alpha}
 +\frac{4\norm[L^1]{w}}{\alpha}\right).
\]
Apply this estimate to $v$ and $v'$.  Approximation by smooth functions in
$W^{3,1}(\R)$ gives $\partial_x\mathcal Av=\mathcal A(v')$ and proves the
first assertion.  If $\operatorname{supp}v\subset[d,\infty)$ and
$x\in(-1,0)$, then $v(x)=0$, which gives the displayed integral formula.
Integrating that formula and its derivative over $(-1,0)$ proves the second
assertion.
\end{proof}

\begin{proposition}
\label{prop:qmode}
For every fixed $0<\alpha<2$, one has
$\wt\phi_n\in\Dom(\mathcal A_I)$.  Set
$z_n=(\mathcal A_I-\mu_n^\alpha)\wt\phi_n$.  Then
\begin{equation}\label{eq:qnorm-sharp}
 \norm[L^\infty(I)]{\wt\phi_n}\le C_\alpha,
 \qquad
 \norm[L^2(I)]{\wt\phi_n}^2=1+o_\alpha(\mu_n^{-1}).
\end{equation}
Moreover,
\begin{equation}\label{eq:zBV}
 \norm[L^2(I)]{z_n}+\norm[L^\infty(I)]{z_n}+\Var_I(z_n)
 \le C_\alpha\mu_n^{-1}.
\end{equation}
As $n\to\infty$,
\begin{equation}\label{eq:zlimit}
 \norm[W^{1,1}(-1,0)]{\mu_nz_n-R_{\alpha,n}}
 +\norm[L^2(I)]{\mu_nz_n-R_{\alpha,n}}
 =o_\alpha(1).
\end{equation}
If $\varphi_k$ is the normalized eigenfunction for
$\lambda_k(\alpha)$, then
\begin{equation}\label{eq:zk}
 \abs{\ip{z_n}{\varphi_k}}\le\frac{C_\alpha}{\mu_n k},
 \qquad k,n\ge1.
\end{equation}
After choosing the sign of $\varphi_n$ and writing
$u_n=\wt\phi_n/\norm[L^2(I)]{\wt\phi_n}$,
\begin{equation}\label{eq:qmode-L2close}
 \norm[L^2(I)]{u_n-\varphi_n}\le C_\alpha n^{-1-\alpha}.
\end{equation}
Finally,
\begin{equation}\label{eq:ordinary-transfer}
 \left|\frac{\calE_{\alpha,I}[\wt\phi_n]}
 {\norm[L^2(I)]{\wt\phi_n}^2}-\lambda_n\right|
 \le C_\alpha\mu_n^{-2-\alpha}.
\end{equation}
\end{proposition}

\begin{proof}
By \cite[Lemma~1]{KwasnickiInterval} and
\cite[Lemma~3.3]{KaletaKwasnickiMalecki},
\[
 \wt\phi_n\in\Dom(\mathcal A_I),
 \qquad
 \mathcal A_I\wt\phi_n=(\mathcal A\wt\phi_n)|_I
 \quad\text{a.e. on }I.
\]

\paragraph{\textbf{Norm of the quasimode.}}
Since $0\le G_\alpha\le\sin\theta$, \eqref{eq:qmode-decomp} gives
$\norm[L^\infty(I)]{\wt\phi_n}\le C_\alpha$.  Equation
\eqref{eq:Gmass} also gives
\[
 \norm[L^1(I)]{\wt\phi_n-S_n}\le \frac{2M_\alpha}{\mu_n},
 \qquad
 \norm[L^2(I)]{S_n}^2=1+\frac{\sin(2\theta)}{2\mu_n}.
\]
Moreover,
\[
 \left|\norm[L^2(I)]{\wt\phi_n}^2-\norm[L^2(I)]{S_n}^2\right|
 \le (\norm[L^\infty]{\wt\phi_n}+\norm[L^\infty]{S_n})
 \norm[L^1]{\wt\phi_n-S_n}
 \le C_\alpha\mu_n^{-1}.
\]
Thus
\begin{equation}\label{eq:qnorm-rough}
 \norm[L^2(I)]{\wt\phi_n}=1+O_\alpha(\mu_n^{-1}),
 \qquad
 \sup_{x\in I}\left|\int_{-1}^x\wt\phi_n(t)\,dt\right|
 \le C_\alpha\mu_n^{-1}.
\end{equation}
Indeed, the antiderivative $x\mapsto\int_{-1}^xS_n(t)\,dt$ is bounded by $2/\mu_n$, and the other two
terms in \eqref{eq:qmode-decomp} have total $L^1$ norm at most
$2M_\alpha/\mu_n$.

For the sharper norm estimate, the change of variables $t=\mu_n(1+x)$ and
the tail estimates give
\[
 \begin{aligned}
 &\int_I q(-x)\left[F_\alpha(\mu_n(1+x))^2
 -\sin^2(\mu_n(1+x)+\theta)\right]dx\\
 &\qquad=\mu_n^{-1}\int_0^\infty
 \left[F_\alpha(t)^2-\sin^2(t+\theta)\right]dt
 +O_\alpha(\mu_n^{-1-\alpha}).
 \end{aligned}
\]
The integral obtained by the change of variables $x\mapsto-x$ has the same
value.  The cross term satisfies
\[
 \int_Iq(x)q(-x)\bigl[G_\alpha(\mu_n(1+x))
 -(-1)^{n+1} G_\alpha(\mu_n(1-x))\bigr]^2\,dx
 =O_\alpha(\mu_n^{-2-2\alpha}).
\]
Hence
\eqref{eq:halfline-norm-identity} yields
\[
 \norm[L^2(I)]{\wt\phi_n}^2
 =1+\frac{\sin(2\theta)}{2\mu_n}
 -\frac{\sin(2\theta)}{2\mu_n}+o_\alpha(\mu_n^{-1}),
\]
which proves \eqref{eq:qnorm-sharp}.

\medskip

\paragraph{\textbf{Residual estimate and limit.}}
On $\R$, set
\[
 a_n(y)=q(y)\one_{(-1,\infty)}(y)G_\alpha(\mu_n(1+y)),
 \qquad
 b_n(y)=q(y)\one_{(-\infty,1)}(y)G_\alpha(\mu_n(1-y)),
\]
\[
 g_n=a_n-(-1)^{n+1} b_n,
 \qquad h_n=S_n\one_{[1,\infty)}.
\]
By \eqref{eq:match},
$\wt\phi_n=F_\alpha(\mu_n(1+\cdot))+g_n-h_n$ on $\R$, and therefore
\begin{equation}\label{eq:error-decomp}
 z_n=\mathcal A g_n-\mu_n^\alpha g_n-\mathcal A h_n
 \quad\text{a.e. on }I.
\end{equation}
The differentiated tail estimates and \cref{lem:operator-bound} give
\[
 \norm[W^{1,1}(-1,0)]{\mathcal Aa_n}
 +\norm[W^{1,1}(-1,0)]{\mathcal A((1-\chi)b_n)}
 \le C_\alpha\mu_n^{-1-\alpha}.
\]
For $x\in(-1,0)$, the support of $\chi b_n$ is separated from $x$, and
$t=\mu_n(1-y)$ gives
\[
 \mu_n\mathcal A(\chi b_n)(x)
 =-c_\alpha\int_0^\infty
 \chi(1-t/\mu_n)G_\alpha(t)(1-x-t/\mu_n)^{-1-\alpha}\,dt.
\]
The integrand and its $x$-derivative are dominated by an integrable multiple
of $G_\alpha(t)$.  Dominated convergence therefore gives
\[
 \mu_n\mathcal Ab_n\longrightarrow
 -c_\alpha M_\alpha(1-x)^{-1-\alpha}
 \quad\text{in }W^{1,1}(-1,0).
\]
The differentiated tail expansion gives
\[
 \norm[W^{1,1}(-1,0)]{
 -\mu_n^{1+\alpha}g_n
 +K_\alpha q(x)(1+x)^{-1-\alpha}
 -(-1)^{n+1}K_\alpha q(x)(1-x)^{-1-\alpha}}
 \longrightarrow0.
\]

Since $S_n''=-\mu_n^2S_n$, $S_n(1)=(-1)^{n+1}\sin\theta$, and
$S_n'(1)=-(-1)^{n+1}\mu_n\cos\theta$, two integrations by parts give
\begin{equation}\label{eq:Ah-two-parts}
\begin{aligned}
 \int_1^\infty\frac{S_n(y)}{(y-x)^{1+\alpha}}\,dy
 &=-\frac{(-1)^{n+1}\cos\theta}{\mu_n}(1-x)^{-1-\alpha}
 +\frac{(-1)^{n+1}(1+\alpha)\sin\theta}{\mu_n^2}(1-x)^{-2-\alpha}\\
 &\quad-\frac{(1+\alpha)(2+\alpha)}{\mu_n^2}
 \int_1^\infty\frac{S_n(y)}{(y-x)^{3+\alpha}}\,dy.
\end{aligned}
\end{equation}
The last integral satisfies
\[
 \left\|\int_1^\infty
 \frac{S_n(y)}{(y-x)^{3+\alpha}}\,dy\right\|_{W_x^{1,1}(-1,0)}
 \le C_\alpha.
\]
Since $\mathcal Ah_n$ is $-c_\alpha$ times the left-hand side,
\[
 \norm[W^{1,1}(-1,0)]{
 \mu_n\mathcal Ah_n
 -(-1)^{n+1}c_\alpha\cos\theta\,(1-x)^{-1-\alpha}}
 \longrightarrow0.
\]
Substituting these estimates into \eqref{eq:error-decomp}, and using
\[
 c_\alpha(M_\alpha-\cos\theta)=-K_\alpha,
 \qquad q(x)-1=-q(-x),
\]
gives
\[
 \norm[W^{1,1}(-1,0)]{\mu_nz_n-R_{\alpha,n}}
 =o_\alpha(1).
\]
The estimate for
$\mathcal Aa_n+\mathcal A((1-\chi)b_n)$ and the preceding $W^{1,1}$ estimates give
$\norm[W^{1,1}(-1,0)]{z_n}\le C_\alpha\mu_n^{-1}$ for all large $n$.
Increasing $C_\alpha$ covers the remaining indices.

The symmetry $z_n(-x)=(-1)^{n+1}z_n(x)$ gives the $L^2(I)$ and
$L^\infty(I)$ bounds, while
\[
 \Var_I(z_n)\le2\Var_{(-1,0)}(z_n)+2|z_n(0-)|
 \le C\norm[W^{1,1}(-1,0)]{z_n}.
\]
This proves \eqref{eq:zBV}.  Equations \eqref{eq:R-reflection} and
$z_n(-x)=(-1)^{n+1}z_n(x)$, together with
$W^{1,1}(-1,0)\hookrightarrow L^2(-1,0)$, give the $L^2(I)$ term in
\eqref{eq:zlimit}.

\medskip

\paragraph{\textbf{Approximation of the eigenfunction.}}
Theorem~3 of \cite{NazarovComparison} and Theorem~2 of
\cite{MusinaNazarov}, followed by the min--max principle, give
\begin{equation}\label{eq:sandwich}
 \left(\frac{(k-1)\pi}{2}\right)^\alpha
 \le\lambda_k(\alpha)\le
 \left(\frac{k\pi}{2}\right)^\alpha.
\end{equation}
Consequently, for $k\ne n$ and all sufficiently large $n$,
\begin{equation}\label{eq:gaps}
 \abs{\lambda_k-\mu_n^\alpha}\ge C_\alpha^{-1}
 \begin{cases}
  \mu_n^{\alpha-1}\abs{k-n},&n/2\le k\le2n,\\
  \mu_n^\alpha,&k<n/2\text{ or }k>2n.
 \end{cases}
\end{equation}
Write $u_n=\sum_{k\ge1}a_{kn}\varphi_k$ with $a_{nn}\ge0$, and set
\[
 \delta_n=\norm[L^2(I)]{u_n-\varphi_n},
 \qquad
 P_k=\sup_{x\in I}\left|\int_{-1}^x\varphi_k(t)\,dt\right|.
\]
For $k\ne n$,
\[
 a_{kn}=\frac{\ip{z_n}{\varphi_k}}
 {\norm[L^2]{\wt\phi_n}(\lambda_k-\mu_n^\alpha)}.
\]
Since $u_n$ and $\varphi_n$ are normalized and $a_{nn}\ge0$,
\[
 \delta_n^2=2(1-a_{nn})\le2\sum_{k\ne n}|a_{kn}|^2.
\]
Using \eqref{eq:gaps} and
$\norm[L^2]{\wt\phi_n}\ge1/2$ for large $n$, we obtain
\begin{align*}
 \sum_{k\ne n}|a_{kn}|^2
 &\le C_\alpha\mu_n^{2-2\alpha}
   \sum_{\substack{n/2\le k\le2n\\k\ne n}}
   |\ip{z_n}{\varphi_k}|^2
 +C_\alpha\mu_n^{-2\alpha}
   \sum_{k<n/2\,\text{or}\,k>2n}|\ip{z_n}{\varphi_k}|^2\\
 &\le C_\alpha\left(\mu_n^{2-2\alpha}+\mu_n^{-2\alpha}\right)
 \norm[L^2(I)]{z_n}^2
 \le C_\alpha n^{-2\alpha},
\end{align*}
where the second line uses Parseval and \eqref{eq:zBV}.  Hence
$\delta_n\le C_\alpha n^{-\alpha}$.

The antiderivative estimate in \eqref{eq:qnorm-rough} and Cauchy--Schwarz give
\[
 P_k\le \sup_{x\in I}\left|\int_{-1}^x u_k(t)\,dt\right|
 +\norm[L^1(I)]{u_k-\varphi_k}
 \le C_\alpha k^{-1}+\sqrt2\,\delta_k.
\]
Hence
\begin{equation}\label{eq:Pk-delta}
 P_k\le C_\alpha(k^{-1}+\delta_k).
\end{equation}
Integration by parts for $z_n\in BV(I)$ gives
\[
 \abs{\ip{z_n}{\varphi_k}}
 \le P_k\bigl(\abs{z_n(1-)}+\Var_I(z_n)\bigr).
\]
Together with \eqref{eq:zBV}, this yields
\begin{equation}\label{eq:zk-preliminary}
 \abs{\ip{z_n}{\varphi_k}}
 \le C_\alpha n^{-1}(k^{-1}+\delta_k).
\end{equation}
Assume that $\delta_k\le C_\alpha k^{-\gamma}$ for all $k$.  For
$n/2\le k\le2n$, \eqref{eq:zk-preliminary} and the first line of
\eqref{eq:gaps} give
\[
 \begin{aligned}
 \sum_{\substack{n/2\le k\le2n\\k\ne n}}|a_{kn}|^2
 &\le C_\alpha n^{-2}
 \sum_{\substack{n/2\le k\le2n\\k\ne n}}
 \frac{k^{-2\min\{\gamma,1\}}}
 {\mu_n^{2\alpha-2}|k-n|^2}\\
 &\le C_\alpha n^{-2\alpha-2\min\{\gamma,1\}}.
 \end{aligned}
\]
In the two complementary ranges, the second line of \eqref{eq:gaps} and
Parseval give
\[
 \sum_{k<n/2\,\text{or}\,k>2n}|a_{kn}|^2
 \le C_\alpha\mu_n^{-2\alpha}
 \sum_{k\ge1}|\ip{z_n}{\varphi_k}|^2
 \le C_\alpha n^{-2-2\alpha}.
\]
Therefore
\begin{equation}\label{eq:delta-iteration}
 \delta_n\le C_\alpha\left(
 n^{-\alpha-\min\{\gamma,1\}}+n^{-1-\alpha}\right).
\end{equation}
If $\alpha\ge1$, the initial estimate already gives
$\delta_n\le C_\alpha n^{-1}$.  If $0<\alpha<1$, then each application
of \eqref{eq:delta-iteration}, while the current exponent is below $1$,
adds $\alpha$ to that exponent.  Applying it
$\lceil1/\alpha\rceil-1$ times therefore gives
$\delta_n\le C_\alpha n^{-1}$.  Thus
$\delta_k\le C_\alpha k^{-1}$ for all sufficiently large $k$, and the
constant can be enlarged to cover the finitely many remaining $k$.
Equations \eqref{eq:Pk-delta} and \eqref{eq:zk-preliminary} now prove
\eqref{eq:zk}.  One more application of \eqref{eq:delta-iteration} gives
\eqref{eq:qmode-L2close} for large $n$, and enlarging $C_\alpha$ covers the
remaining finite set of indices.

\medskip
\paragraph{\textbf{Rayleigh quotient.}}
Taking the inner product of
$z_n=(\mathcal A_I-\mu_n^\alpha)\wt\phi_n$ with $\varphi_n$ gives
\[
 (\lambda_n-\mu_n^\alpha)a_{nn}
 =\frac{\ip{z_n}{\varphi_n}}{\norm[L^2(I)]{\wt\phi_n}}.
\]
Since $a_{nn}\to1$, \eqref{eq:zk} with $k=n$ gives
\begin{equation}\label{eq:lambda-mu-close}
 \abs{\lambda_n-\mu_n^\alpha}=O_\alpha(\mu_n^{-2}).
\end{equation}
For $k\ne n$ and large $n$, this estimate and \eqref{eq:gaps} imply
$|\lambda_k-\lambda_n|\le
C_\alpha|\lambda_k-\mu_n^\alpha|$.  Expanding in the eigenbasis gives
\[
 \left|\frac{\calE_{\alpha,I}[\wt\phi_n]}
 {\norm[L^2]{\wt\phi_n}^2}-\lambda_n\right|
 \le C_\alpha\sum_{k\ne n}
 \frac{|\ip{z_n}{\varphi_k}|^2}{|\lambda_k-\mu_n^\alpha|}.
\]
The range $n/2\le k\le2n$ contributes
$O_\alpha(n^{-3-\alpha}\log n)$ by \eqref{eq:zk} and the first line of
\eqref{eq:gaps}.  Parseval, \eqref{eq:zBV}, and the second line of
\eqref{eq:gaps} bound the two complementary ranges by
$O_\alpha(n^{-2-\alpha})$.  Thus the last display is at most
$C_\alpha n^{-2-\alpha}$.
This proves \eqref{eq:ordinary-transfer}.
\end{proof}

\begin{theorem}[Two-term asymptotic formula]\label{thm:two-term}
For every fixed $0<\alpha<2$,
\begin{equation}\label{eq:third}
 \lambda_n=\mu_n^\alpha+(-1)^n\kappa_\alpha\mu_n^{-2}
 +o_\alpha(\mu_n^{-2}).
\end{equation}
\end{theorem}

\begin{proof}
Equation \eqref{eq:zlimit} gives
\[
 \norm[W^{1,1}(-1,0)]{
 z_n-\mu_n^{-1}R_{\alpha,n}}=o_\alpha(\mu_n^{-1}).
\]
Integration by parts gives
\[
 \begin{aligned}
 \int_{-1}^0R_{\alpha,n}(x)S_n(x)\,dx
 &=\frac{R_{\alpha,n}(-1)\cos\theta-R_{\alpha,n}(0)\cos(\mu_n+\theta)}{\mu_n}\\
 &\quad+\frac1{\mu_n}\int_{-1}^0
 R_{\alpha,n}'(x)\cos(\mu_n(1+x)+\theta)\,dx.
 \end{aligned}
\]
By \eqref{eq:R-reflection},
$R_{\alpha,n}(0)=(-1)^{n+1}R_{\alpha,n}(0)$, while
\[
 (-1)^{n+1}\cos(\mu_n+\theta)
 =-\cos(\mu_n+\theta).
\]
Hence
$R_{\alpha,n}(0)\cos(\mu_n+\theta)=0$.
Since $R_{\alpha,n+2}=R_{\alpha,n}$, the Riemann--Lebesgue lemma applied
to the two functions in this family gives
\[
 \int_{-1}^0R_{\alpha,n}'(x)
 \cos(\mu_n(1+x)+\theta)\,dx=o_\alpha(1).
\]
A second integration by parts applied to
$z_n-\mu_n^{-1}R_{\alpha,n}$ therefore gives
\begin{equation}\label{eq:S-pairing-two-term}
 \int_{-1}^0z_n(x)S_n(x)\,dx
 =\mu_n^{-2}R_{\alpha,n}(-1)\cos\theta
 +o_\alpha(\mu_n^{-2}).
\end{equation}
Furthermore, the substitution $t=\mu_n(1+x)$ gives
\[
 \begin{aligned}
 &\int_{-1}^0z_n(x)q(-x)G_\alpha(\mu_n(1+x))\,dx\\
 &\quad=\mu_n^{-2}\int_0^{\mu_n}
 \mu_nz_n(-1+t/\mu_n)q(1-t/\mu_n)G_\alpha(t)\,dt\\
 &\quad=\mu_n^{-2}R_{\alpha,n}(-1)M_\alpha
 +o_\alpha(\mu_n^{-2}).
 \end{aligned}
\]
Here \eqref{eq:zlimit}, the identity
$R_{\alpha,n+2}=R_{\alpha,n}$, and $G_\alpha\in L^1(0,\infty)$ justify
dominated convergence.  Since $q(x)=0$ for $x\le-1/3$,
\[
 \left|\int_{-1}^0z_n(x)q(x)G_\alpha(\mu_n(1-x))\,dx\right|
 \le C_\alpha\mu_n^{-2-\alpha}
\]
by \eqref{eq:zBV} and the tail estimate for $G_\alpha$.

The identities
\[
 z_n(-x)=(-1)^{n+1}z_n(x),
 \qquad
 \wt\phi_n(-x)=(-1)^{n+1}\wt\phi_n(x),
\]
together with \eqref{eq:qmode-decomp} and
\eqref{eq:S-pairing-two-term}, now give
\[
 \ip{z_n}{\wt\phi_n}
 =2\mu_n^{-2}R_{\alpha,n}(-1)
 (\cos\theta-M_\alpha)+o_\alpha(\mu_n^{-2}).
\]
By \eqref{eq:R-alpha-n},
\[
 R_{\alpha,n}(-1)
 =-(-1)^{n+1}K_\alpha2^{-1-\alpha},
 \qquad
 \cos\theta-M_\alpha=\sqrt{\alpha/2}.
\]
Consequently,
\begin{equation}\label{eq:third-pairing}
 \ip{z_n}{\wt\phi_n}
 =(-1)^n\kappa_\alpha\mu_n^{-2}
 +o_\alpha(\mu_n^{-2}).
\end{equation}
Finally,
\[
 \frac{\calE_{\alpha,I}[\wt\phi_n]}{\norm[L^2(I)]{\wt\phi_n}^2}-\mu_n^\alpha
 =\frac{\ip{z_n}{\wt\phi_n}}{\norm[L^2(I)]{\wt\phi_n}^2}.
\]
Equations \eqref{eq:qnorm-sharp}, \eqref{eq:ordinary-transfer}, and
\eqref{eq:third-pairing} prove \eqref{eq:third}.
\end{proof}
\section{The range \texorpdfstring{$1\le\alpha<2$}{1 <= alpha < 2}}
\label{sec:large-unified}

This section proves \eqref{eq:main-critical} and \eqref{eq:main-large}. The strategy in Section \ref{sec:glued} still applies.
At $\alpha=1$ we determine the $\mu_n^{-2}\log\mu_n$ term in $z_n$, and
for $1<\alpha<2$ we determine its $\mu_n^{-2}$ term.  Then we evaluate
$\langle z_n,\wt\phi_n\rangle$ to obtain the third term in the asymptotic formula.

At $\alpha=1$,
\[
 K_1=\frac1{\pi\sqrt2},
 \qquad
 M_1=\cos\frac\pi8-\frac1{\sqrt2}
 =\int_0^\infty G_1(t)\,dt.
\]

\begin{proposition}\label{prop:alpha-ge-one}
Equations \eqref{eq:main-critical} and \eqref{eq:main-large} hold.
\end{proposition}

\begin{proof}

\medskip
\noindent\textbf{Expansion of the residual.}\quad
By \eqref{eq:error-decomp}, on $(-1,0)$,
\[
 z_n=\mathcal Aa_n-(-1)^{n+1}\mathcal Ab_n-\mu_n^\alpha g_n-\mathcal Ah_n.
\]
The differentiated tail estimates and \cref{lem:operator-bound} give
\[
 \norm[W^{1,1}(-1,0)]{\mathcal Aa_n}
 +\norm[W^{1,1}(-1,0)]{\mathcal A((1-\chi)b_n)}
 \le
 \begin{cases}
  C\mu_n^{-2},&\alpha=1,\\
  C_\alpha\mu_n^{-1-\alpha}=o_\alpha(\mu_n^{-2}),&1<\alpha<2.
 \end{cases}
\]
For $x\in(-1,0)$, the substitution $t=\mu_n(1-y)$ gives
\[
 \mathcal A(\chi b_n)(x)
 =-\frac{c_\alpha}{\mu_n}\int_0^\infty
 \frac{\chi(1-t/\mu_n)G_\alpha(t)}{(1-x-t/\mu_n)^{1+\alpha}}\,dt.
\]
On the support of the integrand, $1-x-t/\mu_n\ge1/3$, and Taylor's formula gives
\[
 \begin{aligned}
 (1-x-t/\mu_n)^{-1-\alpha}
 &=(1-x)^{-1-\alpha}
 +(1+\alpha)\frac{t}{\mu_n}(1-x)^{-2-\alpha}\\
 &\quad+\frac{(1+\alpha)(2+\alpha)t^2}{\mu_n^2}
 \int_0^1\frac{1-r}{(1-x-rt/\mu_n)^{3+\alpha}}\,dr.
 \end{aligned}
\]
Because $\chi(1-t/\mu_n)=1$ for $t\le\mu_n/2$ and vanishes for
$t\ge2\mu_n/3$, \eqref{eq:Gcritical}--\eqref{eq:Glarge} imply
\[
 \begin{aligned}
 \int_0^\infty\chi(1-t/\mu_n)G_\alpha(t)\,dt
 &=M_\alpha+
  O_\alpha(\mu_n^{-\alpha}),
  \ 1\le\alpha<2,\\
 \int_0^\infty t\chi(1-t/\mu_n)G_\alpha(t)\,dt
 &=
 \begin{cases}
  K_1\log\mu_n+O(1),&\alpha=1,\\
  M_{1,\alpha}+O_\alpha(\mu_n^{1-\alpha}),&1<\alpha<2,
 \end{cases}\\
 \int_0^\infty t^2\chi(1-t/\mu_n)G_\alpha(t)\,dt
 &=
  O_\alpha(\mu_n^{2-\alpha}),\ 1\le\alpha<2. 
 \end{aligned}
\]
The Taylor remainder is therefore $O(\mu_n^{-2})$ in
$W^{1,1}(-1,0)$ when $\alpha=1$ and
$O_\alpha(\mu_n^{-1-\alpha})=o_\alpha(\mu_n^{-2})$ when $\alpha>1$.
Together with the estimate for $(1-\chi)b_n$, this gives in
$W^{1,1}(-1,0)$
\[
 \mathcal Ab_n(x)=
 \begin{cases}
 -c_1M_1\mu_n^{-1}(1-x)^{-2}
 -2c_1K_1\mu_n^{-2}\log\mu_n\,(1-x)^{-3}
 +O(\mu_n^{-2}),&\alpha=1,\\[1mm]
 -c_\alpha M_\alpha\mu_n^{-1}(1-x)^{-1-\alpha}
 -c_\alpha(1+\alpha)M_{1,\alpha}\mu_n^{-2}(1-x)^{-2-\alpha}
 +o_\alpha(\mu_n^{-2}),&1<\alpha<2.
 \end{cases}
\]
The tail expansions also give in $W^{1,1}(-1,0)$
\[
 -\mu_n^\alpha g_n=
 \begin{cases}
 -K_1\mu_n^{-1}q(x)\bigl((1+x)^{-2}-(-1)^{n+1}(1-x)^{-2}\bigr)\\
 \qquad+\dfrac{2K_1}{\pi}\mu_n^{-2}\log\mu_n\,
 q(x)\bigl((1+x)^{-3}-(-1)^{n+1}(1-x)^{-3}\bigr)
 +O(\mu_n^{-2}),&\alpha=1,\\[2mm]
 -K_\alpha\mu_n^{-1}q(x)\bigl((1+x)^{-1-\alpha}-(-1)^{n+1}(1-x)^{-1-\alpha}\bigr)\\
 \qquad-K_{2,\alpha}\mu_n^{-2}q(x)
 \bigl((1+x)^{-2-\alpha}-(-1)^{n+1}(1-x)^{-2-\alpha}\bigr)
 +o_\alpha(\mu_n^{-2}),&1<\alpha<2.
 \end{cases}
\]
For $\alpha=1$, the terms containing $\log(1\pm x)$ are
$O(\mu_n^{-2})$ in $W^{1,1}(-1,0)$ because $q$ vanishes near $x=-1$.

A further integration by parts in the last integral of
\eqref{eq:Ah-two-parts} gives
\[
 \int_1^\infty\frac{S_n(y)}{(y-x)^{3+\alpha}}\,dy
 =-\frac{(-1)^{n+1}\cos\theta}{\mu_n}(1-x)^{-3-\alpha}
 -\frac{3+\alpha}{\mu_n}\int_1^\infty
 \frac{\cos(\mu_n(1+y)+\theta)}{(y-x)^{4+\alpha}}\,dy.
\]
The last integral and its $x$-derivative have bounded $L^1(-1,0)$ norms.
Hence, in $W^{1,1}(-1,0)$,
\[
 \begin{aligned}
 \mathcal Ah_n(x)
 ={}&(-1)^{n+1}c_\alpha\cos\theta\,
 \mu_n^{-1}(1-x)^{-1-\alpha}\\
 &-(-1)^{n+1}c_\alpha(1+\alpha)\sin\theta\,
 \mu_n^{-2}(1-x)^{-2-\alpha}+O_\alpha(\mu_n^{-3}).
 \end{aligned}
\]
Substituting these formulas into \eqref{eq:error-decomp} and using $c_1=1/\pi$ and
\[
 c_\alpha(M_\alpha-\cos\theta)=-K_\alpha,
 \qquad q(x)-1=-q(-x),
\]
together with
\[
 c_\alpha(1+\alpha)(M_{1,\alpha}+\sin\theta)=-K_{2,\alpha}
 \qquad(1<\alpha<2),
\]
yields the following two expansions in $W^{1,1}(-1,0)$:
\[
 z_n=\mu_n^{-1}R_{1,n}
 +\frac{2K_1}{\pi}\mu_n^{-2}\log\mu_n
 \bigl(q(x)(1+x)^{-3}+(-1)^{n+1} q(-x)(1-x)^{-3}\bigr)
 +O(\mu_n^{-2})
\]
when $\alpha=1$, and
\[
 z_n=\mu_n^{-1}R_{\alpha,n}
 -K_{2,\alpha}\mu_n^{-2}
 \bigl(q(x)(1+x)^{-2-\alpha}
 +(-1)^{n+1} q(-x)(1-x)^{-2-\alpha}\bigr)
 +o_\alpha(\mu_n^{-2})
\]
when $1<\alpha<2$.

\medskip
\noindent\textbf{Pairing with the quasimode.}\quad
Let $P_n$ denote $R_{\alpha,n}$ or one of the bracketed coefficient
functions in the two residual expansions above.  Each such family satisfies
\[
 P_{n+2}=P_n,
 \qquad
 P_n(-x)=(-1)^{n+1}P_n(x),
 \qquad
 P_n''\in L^1(-1,0).
\]
Two integrations by parts give
\[
 \begin{aligned}
 \int_{-1}^0P_n(x)S_n(x)\,dx
 &=\frac{P_n(-1)\cos\theta-P_n(0)\cos(\mu_n+\theta)}{\mu_n}\\
 &\quad+\frac{P_n'(0)\sin(\mu_n+\theta)-P_n'(-1)\sin\theta}{\mu_n^2}
 -\frac1{\mu_n^2}\int_{-1}^0P_n''(x)S_n(x)\,dx.
 \end{aligned}
\]
The identity $P_n(-x)=(-1)^{n+1}P_n(x)$ gives
\[
 P_n(0)=(-1)^{n+1}P_n(0),
 \qquad
 P_n'(0)=-(-1)^{n+1}P_n'(0),
\]
while $\mu_n+\theta=n\pi/2$ gives
\[
 (-1)^{n+1}\cos(\mu_n+\theta)=-\cos(\mu_n+\theta),
 \qquad
 (-1)^{n+1}\sin(\mu_n+\theta)=\sin(\mu_n+\theta).
\]
Therefore
\[
 P_n(0)\cos(\mu_n+\theta)=0,
 \qquad
 P_n'(0)\sin(\mu_n+\theta)=0.
\]
Since $P_{n+2}=P_n$, the Riemann--Lebesgue lemma applied to the two
members of each family gives
\[
 \int_{-1}^0P_n(x)S_n(x)\,dx
 =\frac{P_n(-1)\cos\theta}{\mu_n}
 -\frac{P_n'(-1)\sin\theta}{\mu_n^2}+o(\mu_n^{-2}).
\]
The substitution $t=\mu_n(1+x)$ and Taylor's formula at $x=-1$ give
\[
 \int_{-1}^0P_n(x)q(-x)G_1(\mu_n(1+x))\,dx
 =\frac{P_n(-1)M_1}{\mu_n}
 +\frac{P_n'(-1)K_1\log\mu_n}{\mu_n^2}+O(\mu_n^{-2})
\]
when $\alpha=1$, and
\[
 \int_{-1}^0P_n(x)q(-x)G_\alpha(\mu_n(1+x))\,dx
 =\frac{P_n(-1)M_\alpha}{\mu_n}
 +\frac{P_n'(-1)M_{1,\alpha}}{\mu_n^2}+o_\alpha(\mu_n^{-2})
\]
when $1<\alpha<2$.  Indeed, the quadratic Taylor remainder and the part
$2\mu_n/3<t<\mu_n$ are bounded by
\[
 C\mu_n^{-3}\int_0^{2\mu_n/3}t^2G_\alpha(t)\,dt
 +\frac C\mu_n\int_{2\mu_n/3}^\infty G_\alpha(t)\,dt
 +\frac C{\mu_n^2}\int_{2\mu_n/3}^{\mu_n}tG_\alpha(t)\,dt,
\]
which is $O(\mu_n^{-2})$ for $\alpha=1$ and $o_\alpha(\mu_n^{-2})$ for
$1<\alpha<2$.

For $r\in W^{1,1}(-1,0)$,
\[
 \left|\int_{-1}^0r(x)S_n(x)\,dx\right|
 \le\frac C\mu_n\norm[W^{1,1}(-1,0)]{r},
 \qquad
 \left|\int_{-1}^0r(x)q(-x)G_\alpha(\mu_n(1+x))\,dx\right|
 \le\frac{M_\alpha}{\mu_n}\norm[L^\infty(-1,0)]{r}.
\]
Hence the remainder terms in the two residual expansions contribute $O(\mu_n^{-3})$ at $\alpha=1$ and
$o_\alpha(\mu_n^{-3})$ for $1<\alpha<2$.  Moreover, by \eqref{eq:zBV} and the
tail estimate for $G_\alpha$,
\[
 \left|\int_{-1}^0z_n(x)q(x)G_\alpha(\mu_n(1-x))\,dx\right|
 \le C_\alpha\mu_n^{-2-\alpha}.
\]

Suppose first that $\alpha=1$.  The function
\[
 \frac{2K_1}{\pi}\bigl(q(x)(1+x)^{-3}+(-1)^{n+1} q(-x)(1-x)^{-3}\bigr)
\]
has value $(-1)^{n+1} K_1/(4\pi)$ at $x=-1$, while
\[
 R_{1,n}(-1)=R_{1,n}'(-1)=-\frac{(-1)^{n+1}K_1}{4},
 \qquad \cos\theta-M_1=\frac1{\sqrt2}.
\]
Substituting the three preceding pairing identities into \eqref{eq:qmode-decomp}, we obtain
\[
 \begin{aligned}
 \int_{-1}^0z_n(x)\wt\phi_n(x)\,dx
 &=R_{1,n}(-1)(\cos\theta-M_1)\mu_n^{-2}\\
 &\quad+\left[-R_{1,n}'(-1)K_1
 +\frac{(-1)^{n+1} K_1}{4\pi}(\cos\theta-M_1)\right]
 \mu_n^{-3}\log\mu_n+O(\mu_n^{-3})\\
 &=-\frac{(-1)^{n+1}\kappa_1}{2}\mu_n^{-2}
 +\frac{(-1)^{n+1}}{4\pi^2}\mu_n^{-3}\log\mu_n+O(\mu_n^{-3}).
 \end{aligned}
\]
Since $z_n(-x)\wt\phi_n(-x)=z_n(x)\wt\phi_n(x)$,
\[
 \ip{z_n}{\wt\phi_n}
 =-(-1)^{n+1}\kappa_1\mu_n^{-2}
 +\frac{(-1)^{n+1}}{2\pi^2}\mu_n^{-3}\log\mu_n+O(\mu_n^{-3}).
\]

Now let $1<\alpha<2$.  The function
\[
 -K_{2,\alpha}\bigl(q(x)(1+x)^{-2-\alpha}
 +(-1)^{n+1} q(-x)(1-x)^{-2-\alpha}\bigr)
\]
has value $-(-1)^{n+1} K_{2,\alpha}2^{-2-\alpha}$ at $x=-1$, and
\[
 R_{\alpha,n}(-1)=-(-1)^{n+1} K_\alpha2^{-1-\alpha},
 \qquad
 R_{\alpha,n}'(-1)=-(-1)^{n+1}(1+\alpha)K_\alpha2^{-2-\alpha}.
\]
Therefore
\[
 \begin{aligned}
 \int_{-1}^0z_n(x)\wt\phi_n(x)\,dx
 ={}&R_{\alpha,n}(-1)(\cos\theta-M_\alpha)\mu_n^{-2}+\bigl[-R_{\alpha,n}'(-1)(\sin\theta+M_{1,\alpha})\\
 &\qquad-(-1)^{n+1}K_{2,\alpha}2^{-2-\alpha}
 (\cos\theta-M_\alpha)\bigr]\mu_n^{-3}+o_\alpha(\mu_n^{-3}).
 \end{aligned}
\]
Using
\[
 \cos\theta-M_\alpha=\sqrt{\alpha/2},
 \qquad
 \sin\theta+M_{1,\alpha}=-\sqrt{\alpha/2}\cot\frac\pi\alpha,
 \qquad
 K_{2,\alpha}=(1+\alpha)K_\alpha\cot\frac\pi\alpha,
\]
and doubling the half-interval integral, we obtain
\[
 \ip{z_n}{\wt\phi_n}
 =-(-1)^{n+1}\kappa_\alpha\mu_n^{-2}
 -(-1)^{n+1}\eta_\alpha\mu_n^{-3}+o_\alpha(\mu_n^{-3}).
\]

Finally,
\[
 \frac{\calE_{\alpha,I}[\wt\phi_n]}{\norm[L^2(I)]{\wt\phi_n}^2}
 -\mu_n^\alpha
 =\frac{\ip{z_n}{\wt\phi_n}}{\norm[L^2(I)]{\wt\phi_n}^2}.
\]
By \eqref{eq:qnorm-sharp},
$\norm[L^2(I)]{\wt\phi_n}^{-2}=1+o_\alpha(\mu_n^{-1})$.  The
Rayleigh quotient differs from $\lambda_n$ by $O(\mu_n^{-3})$ when
$\alpha=1$ and by $o_\alpha(\mu_n^{-3})$ when $1<\alpha<2$, by
\eqref{eq:ordinary-transfer}.  Since $-(-1)^{n+1}=(-1)^n$, the two pairing
formulas prove \eqref{eq:main-critical} and \eqref{eq:main-large}.
\end{proof}

\begin{remark}
The proof above does not apply when $0<\alpha<1$.  In that range,
$\int_0^{\mu_n}tG_\alpha(t)\,dt$ is of order $\mu_n^{1-\alpha}$, so the linear
Taylor term and the quadratic Taylor remainder contribute at the same order
to $z_n$.  In addition, the error in \eqref{eq:ordinary-transfer} is of order
$\mu_n^{-2-\alpha}$, which is the order of 
\eqref{eq:main-small}.
\end{remark}

\section{The range \texorpdfstring{$0<\alpha<1$}{0 < alpha < 1}}
\label{sec:small}

This section proves \eqref{eq:main-small}. To refine the strategy in Section \ref{sec:glued}, we use an uncut combination of
the two half-line eigenfunctions, and obtain the sharper error \eqref{eq:filtered} using its Rayleigh quotient.

Fix $0<\alpha<1$.  For each $n$, define on $(0,2\mu_n)$
\begin{equation}\label{eq:V}
 V_n(t)=\sin(t+\theta)-G_\alpha(t)
 -(-1)^{n+1}G_\alpha(2\mu_n-t).
\end{equation}
Since $2\mu_n+2\theta=n\pi$,
\begin{equation}\label{eq:V-reflection}
 V_n(2\mu_n-t)=(-1)^{n+1}V_n(t).
\end{equation}
The zero extension of $V_n$ is bounded, compactly supported, and of bounded
variation.  If $u$ has these properties, then, for $|h|\le1$,
\[
 \norm[L^2(\R)]{u(\,\cdot+h)-u}^2
 \le2\norm[L^\infty(\R)]{u}\,
      \norm[L^1(\R)]{u(\,\cdot+h)-u}
 \le2\norm[L^\infty(\R)]{u}\Var_\R(u)|h|.
\]
For $|h|>1$, the left-hand side is at most $4\|u\|_2^2$.  Hence
\eqref{eq:translation-form} and $\int_0^1h^{-\alpha}\,dh<\infty$ imply
$V_n\in H^{\alpha/2}(\R)$.

For $0<\rho<2\mu_n$, define
\[
 H_n(\rho)=\int_0^\infty
 \frac{G_\alpha(2\mu_n+u)}{(u+\rho)^{1+\alpha}}\,du,
 \qquad
 D_n=\calE_{\alpha,(0,2\mu_n)}[V_n]
     -\norm[L^2(0,2\mu_n)]{V_n}^2.
\]

\begin{lemma}\label{lem:exact-form-identity}
For every $\psi\in\wt H^{\alpha/2}(0,2\mu_n)$,
\begin{equation}\label{formid}
 \begin{aligned}
 \calE_{\alpha,(0,2\mu_n)}(V_n,\psi)-\ip{V_n}{\psi}=-c_\alpha\int_0^{2\mu_n}
 \bigl(H_n(2\mu_n-t)+(-1)^{n+1}H_n(t)\bigr)\psi(t)\,dt.
 \end{aligned}
\end{equation}
Consequently,
\begin{equation}\label{eq:Dexact}
\begin{aligned}
 D_n&=-2c_\alpha(-1)^{n+1}\int_0^{2\mu_n}
 F_\alpha(\rho)H_n(\rho)\,d\rho+2c_\alpha\int_0^{2\mu_n}
 G_\alpha(2\mu_n-\rho)H_n(\rho)\,d\rho,\\
 \norm[L^2(0,2\mu_n)]{V_n}^2&=\mu_n+O_\alpha(1).
\end{aligned}
\end{equation}
\end{lemma}

\begin{proof}
Use the same notation for the zero extensions of both copies of $F_\alpha$,
and set
\[
 W_n(t)=F_\alpha(t)+(-1)^{n+1}F_\alpha(2\mu_n-t)-\sin(t+\theta).
\]
On $\R$,
\[
 V_n=W_n
 +(-1)^{n+1}\one_{(-\infty,0)}G_\alpha(2\mu_n-\,\cdot)
 +\one_{(2\mu_n,\infty)}G_\alpha.
\]
For $\psi\in C_c^\infty(0,2\mu_n)$, the half-line eigenvalue equation
\eqref{half}, translation and reflection invariance of $\mathcal A$, and
\[
 \mathcal A\sin(\,\cdot+\theta)=\sin(\,\cdot+\theta)
\]
give $\langle(\mathcal A-1)W_n,\psi\rangle=0$.  Fubini's theorem gives,
for $0<t<2\mu_n$,
\[
 \mathcal A(\one_{(2\mu_n,\infty)}G_\alpha)(t)
 =-c_\alpha H_n(2\mu_n-t),
\]
\[
 \mathcal A\bigl((-1)^{n+1}\one_{(-\infty,0)}
 G_\alpha(2\mu_n-\,\cdot)\bigr)(t)
 =-c_\alpha(-1)^{n+1}H_n(t).
\]
This proves \eqref{formid} for smooth $\psi$.  Since $\alpha<1$, the bound
$H_n(\rho)\le C_\alpha\rho^{-\alpha}$ implies
\[
 H_n(2\mu_n-\cdot)+(-1)^{n+1}H_n
 \in L^{2/(1+\alpha)}(0,2\mu_n).
\]
Moreover, \cite[Theorem~6.5]{DNPV} gives
\[
 \wt H^{\alpha/2}(0,2\mu_n)
 \hookrightarrow L^{2/(1-\alpha)}(0,2\mu_n).
\]
H\"older's inequality therefore makes the right-hand side of \eqref{formid}
continuous on $\wt H^{\alpha/2}(0,2\mu_n)$.  Since
$C_c^\infty(0,2\mu_n)$ is dense in $\wt H^{\alpha/2}(0,2\mu_n)$ by
\cite[Theorem~3.29(ii)]{McLean}, \eqref{formid} follows for every admissible
$\psi$.

Taking $\psi=V_n$, changing $t$ to $2\mu_n-t$ in the first integral, and
using \eqref{eq:V-reflection} give
\[
 D_n=-2c_\alpha(-1)^{n+1}\int_0^{2\mu_n}V_n(t)H_n(t)\,dt.
\]
Since
\[
 V_n(t)=F_\alpha(t)-(-1)^{n+1}G_\alpha(2\mu_n-t),
\]
this is the first identity in \eqref{eq:Dexact}.  Finally,
$G_\alpha\in L^1\cap L^2$ and $|\sin(t+\theta)|\le1$ imply
\[
 \left|\norm[L^2(0,2\mu_n)]{V_n}^2
 -\norm[L^2(0,2\mu_n)]{\sin(\,\cdot+\theta)}^2\right|
 \le C_\alpha.
\]
The last squared norm equals $\mu_n+O_\alpha(1)$.
\end{proof}

Note that the rescaling $t=\mu_n(1+x)$ maps $I=(-1,1)$ onto
$(0,2\mu_n)$. Define
\[
 w_n(x)=V_n(\mu_n(1+x)),\qquad x\in I.
\]
Then
\begin{equation}\label{eq:w}
 w_n(x)=S_n(x)-G_\alpha(\mu_n(1+x))
 -(-1)^{n+1}G_\alpha(\mu_n(1-x)),
\end{equation}
where $S_n$ is defined in \eqref{eq:qmode-decomp}, and
\begin{equation}\label{eq:scaling}
 \norm[L^2(I)]{w_n}^2=\mu_n^{-1}
 \norm[L^2(0,2\mu_n)]{V_n}^2,
 \qquad
 \calE_{\alpha,I}[w_n]=\mu_n^{\alpha-1}
 \calE_{\alpha,(0,2\mu_n)}[V_n].
\end{equation}
Thus $w_n$ is obtained from $\wt\phi_n$ by removing the cutoffs from the two
$G_\alpha$-terms.  This permits the use of the exact quadratic-form identity
\eqref{eq:Dexact}.

\begin{proposition}\label{prop:filtered}
For $0<\alpha<1$,
\begin{equation}\label{eq:filtered}
 \frac{\calE_{\alpha,I}[w_n]}{\norm[L^2(I)]{w_n}^2}
 -\lambda_n=o_\alpha(\mu_n^{-2-\alpha}).
\end{equation}
\end{proposition}

\begin{proof}
Put $d_n=w_n-\wt\phi_n$.  By \eqref{eq:qmode-decomp}, \eqref{eq:w}, and
$q(x)+q(-x)=1$,
\[
 d_n(x)=-q(x)G_\alpha(\mu_n(1+x))
 -(-1)^{n+1}q(-x)G_\alpha(\mu_n(1-x)).
\]
If $q(x)\ne0$, then $1+x\ge2/3$.  If $q(-x)\ne0$, then
$1-x\ge2/3$.  Since $1\pm x\le2$ on $I$, \eqref{eq:Gsmall} and its
differentiated form give
\[
 \sup_{2/3\le s\le2}
 \left|\mu_n^{1+\alpha}G_\alpha(\mu_ns)
 -K_\alpha s^{-1-\alpha}\right|\longrightarrow0
\]
and
\[
 \sup_{2/3\le s\le2}
 \left|\mu_n^{2+\alpha}G_\alpha'(\mu_ns)
 +(1+\alpha)K_\alpha s^{-2-\alpha}\right|\longrightarrow0.
\]
The product rule and \eqref{eq:R-alpha-n} therefore give
\begin{equation}\label{dnlim}
 \norm[W^{1,\infty}(I)]{\mu_n^{1+\alpha}d_n-R_{\alpha,n}}
 =o_\alpha(1).
\end{equation}

Use the same notation for the zero extension of $\mu_n^{1+\alpha}d_n$.
Equation \eqref{dnlim}, the identity $R_{\alpha,n+2}=R_{\alpha,n}$, and an
enlargement of $C_\alpha$ for finitely many indices give
\[
 \norm[L^\infty(\R)]{\mu_n^{1+\alpha}d_n}
 +\norm[L^2(\R)]{\mu_n^{1+\alpha}d_n}
 +\Var_\R(\mu_n^{1+\alpha}d_n)\le C_\alpha.
\]
Indeed,
\[
 \begin{aligned}
 \Var_\R(\mu_n^{1+\alpha}d_n)
 &\le\int_{-1}^1
 \abs{\partial_x(\mu_n^{1+\alpha}d_n)}\,dx\\
 &\quad+\abs{\mu_n^{1+\alpha}d_n(-1+)}
 +\abs{\mu_n^{1+\alpha}d_n(1-)}.
 \end{aligned}
\]
For $0<h\le1$, the bounded-variation estimate preceding
\cref{lem:exact-form-identity} gives
\[
 \begin{aligned}
 \norm[L^2(\R)]{
 \mu_n^{1+\alpha}d_n(\,\cdot+h)-\mu_n^{1+\alpha}d_n}^2\le2h\norm[L^\infty(\R)]{\mu_n^{1+\alpha}d_n}
 \Var_\R(\mu_n^{1+\alpha}d_n).
 \end{aligned}
\]
For $h>1$, the left-hand side is at most
$4\norm[L^2(\R)]{\mu_n^{1+\alpha}d_n}^2$.  Hence
\eqref{eq:translation-form} gives
\[
 \begin{aligned}
 \calE_{\alpha,I}[\mu_n^{1+\alpha}d_n]
 \le \frac{2c_\alpha}{1-\alpha}
 \norm[L^\infty(\R)]{\mu_n^{1+\alpha}d_n}
 \Var_\R(\mu_n^{1+\alpha}d_n)+\frac{4c_\alpha}{\alpha}
 \norm[L^2(\R)]{\mu_n^{1+\alpha}d_n}^2
 \le C_\alpha.
 \end{aligned}
\]
By quadratic homogeneity,
\[
 \sup_n\calE_{\alpha,I}[\mu_n^{1+\alpha}d_n]<\infty,
 \qquad
 \calE_{\alpha,I}[d_n]=O_\alpha(\mu_n^{-2-2\alpha}).
\]

Since $\wt\phi_n\in\Dom(\mathcal A_I)$ and $d_n$ belongs to the form domain,
\[
 \calE_{\alpha,I}(\wt\phi_n,d_n)
 -\mu_n^\alpha\ip{\wt\phi_n}{d_n}
 =\ip{(\mathcal A_I-\mu_n^\alpha)\wt\phi_n}{d_n}
 =\ip{z_n}{d_n}.
\]
Thus
\[
 \begin{aligned}
 &\bigl(\calE_{\alpha,I}[w_n]
 -\mu_n^\alpha\norm[L^2]{w_n}^2\bigr)
 -\bigl(\calE_{\alpha,I}[\wt\phi_n]
 -\mu_n^\alpha\norm[L^2]{\wt\phi_n}^2\bigr)\\
 &\qquad=2\ip{z_n}{d_n}+\calE_{\alpha,I}[d_n]
 -\mu_n^\alpha\norm[L^2]{d_n}^2.
 \end{aligned}
\]
Equations \eqref{eq:zlimit} and \eqref{dnlim} give, in $L^2(I)$,
\[
 z_n=\mu_n^{-1}R_{\alpha,n}+o_\alpha(\mu_n^{-1}),
 \qquad
 d_n=\mu_n^{-1-\alpha}R_{\alpha,n}
 +o_\alpha(\mu_n^{-1-\alpha}).
\]
Consequently,
\[
 \begin{aligned}
 2\ip{z_n}{d_n}
 &=2\mu_n^{-2-\alpha}\norm[L^2(I)]{R_{\alpha,n}}^2
 +o_\alpha(\mu_n^{-2-\alpha}),\\
 \mu_n^\alpha\norm[L^2]{d_n}^2
 &=\mu_n^{-2-\alpha}\norm[L^2(I)]{R_{\alpha,n}}^2
 +o_\alpha(\mu_n^{-2-\alpha}),\\
 \calE_{\alpha,I}[d_n]
 &=O_\alpha(\mu_n^{-2-2\alpha})
 =o_\alpha(\mu_n^{-2-\alpha}).
 \end{aligned}
\]
Therefore
\begin{equation}\label{eq:w-qmode-comparison}
 \begin{aligned}
 &\bigl(\calE_{\alpha,I}[w_n]
 -\mu_n^\alpha\norm[L^2]{w_n}^2\bigr)
 -\bigl(\calE_{\alpha,I}[\wt\phi_n]
 -\mu_n^\alpha\norm[L^2]{\wt\phi_n}^2\bigr)\\
 &\qquad=\mu_n^{-2-\alpha}\norm[L^2(I)]{R_{\alpha,n}}^2
 +o_\alpha(\mu_n^{-2-\alpha}).
 \end{aligned}
\end{equation}

Define
\begin{equation}\label{eq:filter-vector}
 \psi_n=\wt\phi_n+(\mathcal A_I+\mu_n^\alpha)^{-1}z_n.
\end{equation}
For every $k\ge1$,
\[
 \ip{z_n}{\varphi_k}
 =(\lambda_k-\mu_n^\alpha)\ip{\wt\phi_n}{\varphi_k},
 \qquad
 \ip{\psi_n}{\varphi_k}
 =\frac{2\lambda_k}{\lambda_k+\mu_n^\alpha}
 \ip{\wt\phi_n}{\varphi_k}.
\]
Moreover,
\[
 \begin{aligned}
 &(\lambda_k-\mu_n^\alpha)
 \left(\frac{4\lambda_k^2}{(\lambda_k+\mu_n^\alpha)^2}-1\right)
 \abs{\ip{\wt\phi_n}{\varphi_k}}^2\\
 &\qquad=\frac{3\lambda_k+\mu_n^\alpha}
 {(\lambda_k+\mu_n^\alpha)^2}\abs{\ip{z_n}{\varphi_k}}^2,
 \end{aligned}
\]
and
\[
 \frac{3\lambda_k+\mu_n^\alpha}{(\lambda_k+\mu_n^\alpha)^2}
 =\mu_n^{-\alpha}
 +\frac{\lambda_k(\mu_n^\alpha-\lambda_k)}
 {\mu_n^\alpha(\lambda_k+\mu_n^\alpha)^2}.
\]
Finite spectral projections followed by Parseval therefore give
\[
 \begin{aligned}
 &\bigl(\calE_{\alpha,I}[\psi_n]
 -\mu_n^\alpha\norm[L^2]{\psi_n}^2\bigr)
 -\bigl(\calE_{\alpha,I}[\wt\phi_n]
 -\mu_n^\alpha\norm[L^2]{\wt\phi_n}^2\bigr)\\
 &\quad=\mu_n^{-\alpha}\norm[L^2]{z_n}^2
 +\sum_{k\ge1}
 \frac{\lambda_k(\mu_n^\alpha-\lambda_k)}
 {\mu_n^\alpha(\lambda_k+\mu_n^\alpha)^2}
 \abs{\ip{z_n}{\varphi_k}}^2.
 \end{aligned}
\]
For $k\le2n$,
\[
 \left|\frac{\lambda_k(\mu_n^\alpha-\lambda_k)}
 {\mu_n^\alpha(\lambda_k+\mu_n^\alpha)^2}\right|
 \le\lambda_k\mu_n^{-2\alpha}.
\]
Using \eqref{eq:zk}, \eqref{eq:sandwich}, and
$\pi n/4\le\mu_n\le\pi n/2$, we obtain
\[
 \begin{aligned}
 \sum_{k\le2n}
 \left|\frac{\lambda_k(\mu_n^\alpha-\lambda_k)}
 {\mu_n^\alpha(\lambda_k+\mu_n^\alpha)^2}\right|
 \abs{\ip{z_n}{\varphi_k}}^2\le C_\alpha\mu_n^{-2-2\alpha}
 \sum_{k\le2n}k^{\alpha-2}
 =O_\alpha(\mu_n^{-2-2\alpha}).
 \end{aligned}
\]
For $k>2n$,
\[
 \left|\frac{\lambda_k(\mu_n^\alpha-\lambda_k)}
 {\mu_n^\alpha(\lambda_k+\mu_n^\alpha)^2}\right|
 \le\mu_n^{-\alpha},
\]
and \eqref{eq:zk} gives
\[
 \begin{aligned}
 &\sum_{k>2n}
 \left|\frac{\lambda_k(\mu_n^\alpha-\lambda_k)}
 {\mu_n^\alpha(\lambda_k+\mu_n^\alpha)^2}\right|
 \abs{\ip{z_n}{\varphi_k}}^2\\
 &\qquad\le C_\alpha\mu_n^{-2-\alpha}
 \sum_{k>2n}k^{-2}
 =O_\alpha(\mu_n^{-3-\alpha}).
 \end{aligned}
\]
Finally, \eqref{eq:zlimit} gives
\[
 \mu_n^{-\alpha}\norm[L^2]{z_n}^2
 =\mu_n^{-2-\alpha}\norm[L^2(I)]{R_{\alpha,n}}^2
 +o_\alpha(\mu_n^{-2-\alpha}).
\]
It follows that
\begin{equation}\label{eq:psi-qmode-comparison}
 \begin{aligned}
 &\bigl(\calE_{\alpha,I}[\psi_n]
 -\mu_n^\alpha\norm[L^2]{\psi_n}^2\bigr)
 -\bigl(\calE_{\alpha,I}[\wt\phi_n]
 -\mu_n^\alpha\norm[L^2]{\wt\phi_n}^2\bigr)\\
 &\qquad=\mu_n^{-2-\alpha}\norm[L^2(I)]{R_{\alpha,n}}^2
 +o_\alpha(\mu_n^{-2-\alpha}).
 \end{aligned}
\end{equation}
Subtracting \eqref{eq:psi-qmode-comparison} from
\eqref{eq:w-qmode-comparison} gives
\[
 \begin{aligned}
 &\bigl(\calE_{\alpha,I}[w_n]
 -\mu_n^\alpha\norm[L^2]{w_n}^2\bigr)
 -\bigl(\calE_{\alpha,I}[\psi_n]
 -\mu_n^\alpha\norm[L^2]{\psi_n}^2\bigr)\\
 &\qquad=o_\alpha(\mu_n^{-2-\alpha}).
 \end{aligned}
\]

The positivity of $\mathcal A_I$ gives
\[
 \norm[L^2]{\psi_n-\wt\phi_n}
 \le\mu_n^{-\alpha}\norm[L^2]{z_n}
 =O_\alpha(\mu_n^{-1-\alpha}).
\]
Equations \eqref{eq:qnorm-rough}, \eqref{eq:Dexact}, and
\eqref{eq:scaling} give
\[
 \norm[L^2]{\psi_n}^2=1+O_\alpha(\mu_n^{-1}),
 \qquad
 \norm[L^2]{w_n}^2
 =\mu_n^{-1}\bigl(\mu_n+O_\alpha(1)\bigr)
 =1+O_\alpha(\mu_n^{-1}).
\]
Equation \eqref{eq:third-pairing} and the two comparison formulas give
\[
 \abs{\calE_{\alpha,I}[v]-\mu_n^\alpha\norm[L^2]{v}^2}
 \le C_\alpha\mu_n^{-2},
 \qquad v=w_n,\psi_n.
\]
Since both squared norms equal $1+O_\alpha(\mu_n^{-1})$,
\[
 \left|\frac1{\norm[L^2]{w_n}^2}
 -\frac1{\norm[L^2]{\psi_n}^2}\right|
 =O_\alpha(\mu_n^{-1}).
\]
Therefore
\[
 \begin{aligned}
 &\frac{\calE_{\alpha,I}[w_n]}{\norm[L^2]{w_n}^2}
 -\frac{\calE_{\alpha,I}[\psi_n]}{\norm[L^2]{\psi_n}^2}\\
 &=\frac{
 (\calE_{\alpha,I}[w_n]-\mu_n^\alpha\norm[L^2]{w_n}^2)
 -(\calE_{\alpha,I}[\psi_n]-\mu_n^\alpha\norm[L^2]{\psi_n}^2)}
 {\norm[L^2]{w_n}^2}\\
 &\quad+(\calE_{\alpha,I}[\psi_n]
 -\mu_n^\alpha\norm[L^2]{\psi_n}^2)
 \left(\frac1{\norm[L^2]{w_n}^2}
 -\frac1{\norm[L^2]{\psi_n}^2}\right)\\
 &=o_\alpha(\mu_n^{-2-\alpha})+O_\alpha(\mu_n^{-3})
 =o_\alpha(\mu_n^{-2-\alpha}),
 \end{aligned}
\]
because $\alpha<1$.

For all sufficiently large $n$, \eqref{eq:gaps} ensures that
$\lambda_k\ne\mu_n^\alpha$ for $k\ne n$, and
\[
 \ip{\psi_n}{\varphi_k}
 =\frac{2\lambda_k}{\lambda_k+\mu_n^\alpha}
 \frac{\ip{z_n}{\varphi_k}}{\lambda_k-\mu_n^\alpha}.
\]
Equations \eqref{eq:lambda-mu-close} and \eqref{eq:gaps} also give
\[
 \abs{\lambda_k-\lambda_n}
 \le C_\alpha\abs{\lambda_k-\mu_n^\alpha},
 \qquad k\ne n.
\]
Since $\norm[L^2]{\psi_n}^2\ge1/2$ for large $n$, Parseval gives
\[
 \left|\frac{\calE_{\alpha,I}[\psi_n]}{\norm[L^2]{\psi_n}^2}
 -\lambda_n\right|
 \le C_\alpha\sum_{k\ne n}
 \frac{\lambda_k^2}{(\lambda_k+\mu_n^\alpha)^2}
 \frac{\abs{\ip{z_n}{\varphi_k}}^2}
 {\abs{\lambda_k-\mu_n^\alpha}}.
\]
Splitting the sum into $k<n/2$, $n/2\le k\le2n$, and $k>2n$,
\eqref{eq:sandwich}, \eqref{eq:gaps}, and \eqref{eq:zk} give
\[
 \begin{aligned}
 \left|\frac{\calE_{\alpha,I}[\psi_n]}{\norm[L^2]{\psi_n}^2}
 -\lambda_n\right|
 &\le C_\alpha\left(
 \mu_n^{-2-3\alpha}\sum_{k<n/2}k^{2\alpha-2}
 +\mu_n^{-3-\alpha}\log\mu_n
 +\mu_n^{-3-\alpha}\right)\\
 &=o_\alpha(\mu_n^{-2-\alpha}).
 \end{aligned}
\]
Combining this estimate with the comparison of the two Rayleigh quotients
proves \eqref{eq:filtered} for the full sequence.
\end{proof}

\begin{lemma}\label{lem:small-expansions}
Put
\[
 Q_\alpha=\int_0^\infty t^{-\alpha}F_\alpha(t)\,dt.
\]
The integral defining $Q_\alpha$ converges.  Moreover, as $n\to\infty$,
\begin{equation}\label{eq:FH-expansion}
\begin{aligned}
 &\int_0^{2\mu_n}F_\alpha(\rho)H_n(\rho)\,d\rho\\
 &\quad=\frac{K_\alpha Q_\alpha}{\alpha}(2\mu_n)^{-1-\alpha}
 +\frac{K_{2,\alpha}Q_\alpha}{\alpha}(2\mu_n)^{-1-2\alpha}\\
 &\qquad-\frac{K_\alpha\sqrt{\alpha/2}\,
 \Gamma(1-\alpha)\Gamma(1+2\alpha)}
 {\alpha\Gamma(1+\alpha)}(2\mu_n)^{-1-2\alpha}\\
 &\qquad+(-1)^{n+1}\frac{K_\alpha\cos\theta}{1+2\alpha}
 (2\mu_n)^{-1-2\alpha}
 +o_\alpha(\mu_n^{-1-2\alpha})
\end{aligned}
\end{equation}
and
\begin{equation}\label{eq:GH-expansion}
 \begin{aligned}
 \int_0^{2\mu_n}G_\alpha(2\mu_n-\rho)H_n(\rho)\,d\rho=\frac{K_\alpha M_\alpha}{1+2\alpha}
 (2\mu_n)^{-1-2\alpha}
 +o_\alpha(\mu_n^{-1-2\alpha}).
 \end{aligned}
\end{equation}
\end{lemma}

\begin{proof}
By \eqref{eq:F-zero},
\[
 t^{-\alpha}F_\alpha(t)=O_\alpha(t^{-\alpha/2}),
 \qquad 0<t\le1.
\]
Since $0<\alpha<1$, this is integrable at zero.  For $t\ge1$,
\eqref{eq:Gsmall} gives
\[
 t^{-\alpha}G_\alpha(t)=O_\alpha(t^{-1-2\alpha}).
\]
For $1\le A<B$, integration by parts gives
\[
 \left|\int_A^B t^{-\alpha}\sin(t+\theta)\,dt\right|
 \le3A^{-\alpha}.
\]
Thus $\int_1^\infty t^{-\alpha}\sin(t+\theta)\,dt$ converges, and
$Q_\alpha$ is well defined.

The expansion \eqref{eq:Gsmall}, its differentiated form, and complete
monotonicity imply
\begin{equation}\label{eq:G-tail-bounds-small}
 0\le G_\alpha(t)\le C_\alpha t^{-1-\alpha},
 \qquad
 0\le-G_\alpha'(t)\le C_\alpha t^{-2-\alpha},
 \qquad t\ge1.
\end{equation}
For $0<\rho<2\mu_n$,
\begin{equation}\label{eq:H-decomposition}
 \begin{aligned}
 H_n(\rho)
 =\frac{G_\alpha(2\mu_n)}{\alpha\rho^\alpha}+\int_0^\infty
 \frac{G_\alpha(2\mu_n+u)-G_\alpha(2\mu_n)}
 {(u+\rho)^{1+\alpha}}\,du.
 \end{aligned}
\end{equation}
Indeed,
$\int_0^\infty(u+\rho)^{-1-\alpha}\,du=\rho^{-\alpha}/\alpha$.
Moreover,
\[
 \begin{aligned}
 &\int_0^{2\mu_n}|F_\alpha(\rho)|\int_0^\infty
 \frac{|G_\alpha(2\mu_n+u)-G_\alpha(2\mu_n)|}
 {(u+\rho)^{1+\alpha}}\,du\,d\rho\\
 &\qquad\le\frac{G_\alpha(2\mu_n)}\alpha
 \int_0^{2\mu_n}\rho^{-\alpha}|F_\alpha(\rho)|\,d\rho<\infty.
 \end{aligned}
\]
Thus Fubini's theorem may be applied to \eqref{eq:H-decomposition}.  Using
\[
 (u+\rho)^{-1-\alpha}
 =\frac1{\Gamma(1+\alpha)}
 \int_0^\infty s^\alpha e^{-s(u+\rho)}\,ds,
\]
define
\[
 \Phi_n(s)=\int_0^{2\mu_n}e^{-s\rho}F_\alpha(\rho)\,d\rho,
 \qquad
 \Psi_n(s)=\int_0^\infty e^{-su}G_\alpha(2\mu_n+u)\,du.
\]
Then
\begin{equation}\label{eq:FH-Laplace-decomposition}
 \begin{aligned}
 &\int_0^{2\mu_n}F_\alpha(\rho)H_n(\rho)\,d\rho\\
 &\quad=\frac{G_\alpha(2\mu_n)}\alpha
 \int_0^{2\mu_n}\rho^{-\alpha}F_\alpha(\rho)\,d\rho\\
 &\qquad+\frac1{\Gamma(1+\alpha)}\int_0^\infty
 s^\alpha\Phi_n(s)
 \left(\Psi_n(s)-\frac{G_\alpha(2\mu_n)}s\right)ds.
 \end{aligned}
\end{equation}

We first expand the first term in \eqref{eq:FH-Laplace-decomposition}.
Since $2\mu_n+2\theta=n\pi$,
\[
 \cos(2\mu_n+\theta)=\cos(n\pi-\theta)
 =(-1)^n\cos\theta.
\]
Integration by parts gives
\[
 \begin{aligned}
 \int_{2\mu_n}^\infty
 \rho^{-\alpha}\sin(\rho+\theta)\,d\rho&=(2\mu_n)^{-\alpha}\cos(2\mu_n+\theta)
 -\alpha\int_{2\mu_n}^\infty\rho^{-1-\alpha}
 \cos(\rho+\theta)\,d\rho\\
 &=(-1)^n\cos\theta\,(2\mu_n)^{-\alpha}
 +O_\alpha(\mu_n^{-1-\alpha}).
 \end{aligned}
\]
By \eqref{eq:G-tail-bounds-small},
\[
 \begin{aligned}
 0&\le\int_{2\mu_n}^\infty
 \rho^{-\alpha}G_\alpha(\rho)\,d\rho\\
 &\le C_\alpha\int_{2\mu_n}^\infty
 \rho^{-1-2\alpha}\,d\rho
 =O_\alpha(\mu_n^{-2\alpha})
 =o_\alpha(\mu_n^{-\alpha}).
 \end{aligned}
\]
Consequently,
\begin{equation}\label{eq:weighted-F-tail}
 \int_0^{2\mu_n}\rho^{-\alpha}F_\alpha(\rho)\,d\rho
 =Q_\alpha+(-1)^{n+1}\cos\theta\,(2\mu_n)^{-\alpha}
 +o_\alpha(\mu_n^{-\alpha}).
\end{equation}
Equations \eqref{eq:Gsmall} and \eqref{eq:weighted-F-tail} yield
\begin{equation}\label{eq:FH-first-part}
\begin{aligned}
 &\frac{G_\alpha(2\mu_n)}\alpha
 \int_0^{2\mu_n}\rho^{-\alpha}F_\alpha(\rho)\,d\rho\\
 &\quad=\frac{K_\alpha Q_\alpha}{\alpha}(2\mu_n)^{-1-\alpha}+\left(
 \frac{K_{2,\alpha}Q_\alpha}{\alpha}
 +(-1)^{n+1}\frac{K_\alpha\cos\theta}{\alpha}
 \right)(2\mu_n)^{-1-2\alpha}
 +o_\alpha(\mu_n^{-1-2\alpha}).
\end{aligned}
\end{equation}

We next evaluate the second term in \eqref{eq:FH-Laplace-decomposition}.
Direct integration gives
\begin{equation}\label{eq:damped-sine}
 \begin{aligned}
 \int_0^{2\mu_n}e^{-s\rho}\sin(\rho+\theta)\,d\rho
 =\frac{s\sin\theta+\cos\theta
 -e^{-2\mu_ns}\bigl(s\sin(2\mu_n+\theta)
 +\cos(2\mu_n+\theta)\bigr)}{1+s^2}.
 \end{aligned}
\end{equation}
Since $G_\alpha\in L^1(0,\infty)$, \eqref{eq:damped-sine} and
\eqref{eq:Gmass} imply
\begin{equation}\label{eq:Phi-bounds-limit}
 \sup_{\substack{n\ge1\\s>0}}|\Phi_n(s)|\le C_\alpha,
 \qquad
 \Phi_n\left(\frac z{2\mu_n}\right)
 -\sqrt{\frac\alpha2}
 -(-1)^{n+1}\cos\theta\,e^{-z}\longrightarrow0
 \quad(z>0).
\end{equation}
Indeed, the sine integral in \eqref{eq:damped-sine}, with
$s=z/(2\mu_n)$, differs from
$\cos\theta+(-1)^{n+1}\cos\theta e^{-z}$ by $o_\alpha(1)$, while
\[
 \int_0^{2\mu_n}e^{-z\rho/(2\mu_n)}G_\alpha(\rho)\,d\rho
 \longrightarrow\int_0^\infty G_\alpha(\rho)\,d\rho=M_\alpha.
\]

Changing variables $u=2\mu_nv$ gives, for each $z>0$,
\begin{equation}\label{eq:Psi-limit}
 \begin{aligned}
 &(2\mu_n)^\alpha\left(
 \Psi_n\left(\frac z{2\mu_n}\right)
 -\frac{2\mu_nG_\alpha(2\mu_n)}z\right)\\
 &\quad=(2\mu_n)^{1+\alpha}\int_0^\infty e^{-zv}
 \bigl(G_\alpha(2\mu_n(1+v))-G_\alpha(2\mu_n)\bigr)\,dv\\
 &\quad\longrightarrow K_\alpha\int_0^\infty
 \bigl((1+v)^{-1-\alpha}-1\bigr)e^{-zv}\,dv.
 \end{aligned}
\end{equation}
Indeed, for $n\ge1$ and $v\ge0$, \eqref{eq:G-tail-bounds-small} gives
\[
 (2\mu_n)^{1+\alpha}
 \left|G_\alpha(2\mu_n(1+v))-G_\alpha(2\mu_n)\right|
 \le C_\alpha\bigl((1+v)^{-1-\alpha}+1\bigr).
\]
After multiplication by $e^{-zv}$, the right-hand side is integrable in
$v$.  Thus \eqref{eq:Gsmall} and dominated convergence prove
\eqref{eq:Psi-limit}.  We next prove two bounds for its left-hand side.
First,
\[
 \begin{aligned}
 &(2\mu_n)^\alpha\left|
 \Psi_n\left(\frac z{2\mu_n}\right)
 -\frac{2\mu_nG_\alpha(2\mu_n)}z\right|\\
 &\quad\le(2\mu_n)^\alpha\left(
 \Psi_n\left(\frac z{2\mu_n}\right)
 +\frac{2\mu_nG_\alpha(2\mu_n)}z\right)\\
 &\quad\le\frac{2(2\mu_n)^{1+\alpha}G_\alpha(2\mu_n)}z
 \le\frac{C_\alpha}{z}.
 \end{aligned}
\]
Second, integration by parts in $v$ gives
\[
 \Psi_n\left(\frac z{2\mu_n}\right)
 -\frac{2\mu_nG_\alpha(2\mu_n)}z
 =\frac{(2\mu_n)^2}{z}\int_0^\infty
 e^{-zv}G_\alpha'(2\mu_n(1+v))\,dv,
\]
and hence
\[
 \begin{aligned}
 &(2\mu_n)^\alpha\left|
 \Psi_n\left(\frac z{2\mu_n}\right)
 -\frac{2\mu_nG_\alpha(2\mu_n)}z\right|\\
 &\qquad\le\frac{C_\alpha}{z}
 \int_0^\infty e^{-zv}(1+v)^{-2-\alpha}\,dv
 \le\frac{C_\alpha}{z^2}.
 \end{aligned}
\]
Therefore
\begin{equation}\label{eq:Psi-dominating-bound}
 (2\mu_n)^\alpha\left|
 \Psi_n\left(\frac z{2\mu_n}\right)
 -\frac{2\mu_nG_\alpha(2\mu_n)}z\right|
 \le C_\alpha\min\{z^{-1},z^{-2}\}.
\end{equation}
Because $0<\alpha<1$,
$z^\alpha\min\{z^{-1},z^{-2}\}$ is integrable on $(0,\infty)$.
After the change of variables $s=z/(2\mu_n)$, equations
\eqref{eq:Phi-bounds-limit}--\eqref{eq:Psi-dominating-bound} give
\begin{equation}\label{eq:FH-scaled-limit}
 \begin{aligned}
 &(2\mu_n)^{1+2\alpha}
 \frac1{\Gamma(1+\alpha)}\int_0^\infty
 s^\alpha\Phi_n(s)
 \left(\Psi_n(s)-\frac{G_\alpha(2\mu_n)}s\right)ds\\
 &\quad=\frac{K_\alpha}{\Gamma(1+\alpha)}
 \int_0^\infty z^\alpha
 \left(\sqrt{\frac\alpha2}
 +(-1)^{n+1}\cos\theta\,e^{-z}\right)\\
 &\qquad\qquad\times\int_0^\infty
 \bigl((1+v)^{-1-\alpha}-1\bigr)e^{-zv}\,dv\,dz
 +o_\alpha(1).
 \end{aligned}
\end{equation}
The two integrals obtained by reversing the order of integration are
\begin{align}
 \int_0^\infty\bigl((1+v)^{-1-\alpha}-1\bigr)v^{-1-\alpha}\,dv
 &=-\frac{\Gamma(1-\alpha)\Gamma(1+2\alpha)}
 {\alpha\Gamma(1+\alpha)},
 \label{eq:first-constant-integral}\\
 \int_0^\infty\bigl((1+v)^{-1-\alpha}-1\bigr)
 (1+v)^{-1-\alpha}\,dv
 &=\frac1{1+2\alpha}-\frac1\alpha.
 \label{eq:second-constant-integral}
\end{align}
Using
\[
 \int_0^\infty z^\alpha e^{-zv}\,dz
 =\Gamma(1+\alpha)v^{-1-\alpha},
 \qquad
 \int_0^\infty z^\alpha e^{-z(1+v)}\,dz
 =\Gamma(1+\alpha)(1+v)^{-1-\alpha},
\]
equations \eqref{eq:FH-scaled-limit}--\eqref{eq:second-constant-integral}
show that the second term in \eqref{eq:FH-Laplace-decomposition} equals
\begin{equation}\label{eq:FH-second-part}
 \begin{aligned}
 &\left[-\frac{K_\alpha\sqrt{\alpha/2}\,
 \Gamma(1-\alpha)\Gamma(1+2\alpha)}
 {\alpha\Gamma(1+\alpha)}
 +(-1)^{n+1}K_\alpha\cos\theta
 \left(\frac1{1+2\alpha}-\frac1\alpha\right)\right]\\
 &\qquad\times(2\mu_n)^{-1-2\alpha}
 +o_\alpha(\mu_n^{-1-2\alpha}).
 \end{aligned}
\end{equation}
Adding \eqref{eq:FH-first-part} and \eqref{eq:FH-second-part} proves
\eqref{eq:FH-expansion}.

For \eqref{eq:GH-expansion}, the substitutions
$x=2\mu_n-\rho$ and $y=2\mu_n+u$ give
\begin{equation}\label{eq:GH-double-integral}
 \begin{aligned}
 &\int_0^{2\mu_n}G_\alpha(2\mu_n-\rho)H_n(\rho)\,d\rho\\
 &\qquad=\int_0^{2\mu_n}G_\alpha(x)
 \int_{2\mu_n}^\infty
 \frac{G_\alpha(y)}{(y-x)^{1+\alpha}}\,dy\,dx.
 \end{aligned}
\end{equation}
For every fixed $x\ge0$, the change of variables $y=2\mu_nv$ gives
\[
 \begin{aligned}
 (2\mu_n)^{1+2\alpha}\int_{2\mu_n}^\infty
 \frac{G_\alpha(y)}{(y-x)^{1+\alpha}}\,dy=(2\mu_n)^{1+\alpha}\int_1^\infty
 G_\alpha(2\mu_nv)
 \left(v-\frac{x}{2\mu_n}\right)^{-1-\alpha}\,dv.
 \end{aligned}
\]
For all sufficiently large $n$ and $v\ge1$,
\[
 (2\mu_n)^{1+\alpha}G_\alpha(2\mu_nv)
 \left(v-\frac{x}{2\mu_n}\right)^{-1-\alpha}
 \le C_\alpha v^{-2-2\alpha}.
\]
Since the left-hand side converges to
$K_\alpha v^{-2-2\alpha}$, dominated convergence gives
\begin{equation}\label{eq:GH-inner-limit}
 (2\mu_n)^{1+2\alpha}\int_{2\mu_n}^\infty
 \frac{G_\alpha(y)}{(y-x)^{1+\alpha}}\,dy
 \longrightarrow\frac{K_\alpha}{1+2\alpha}.
\end{equation}
If $0\le x\le\mu_n$, then $y-x\ge y/2$ for $y\ge2\mu_n$, and
\begin{equation}\label{eq:GH-inner-bound}
 \begin{aligned}
 (2\mu_n)^{1+2\alpha}\int_{2\mu_n}^\infty
 \frac{G_\alpha(y)}{(y-x)^{1+\alpha}}\,dy\le C_\alpha(2\mu_n)^{1+2\alpha}
 \int_{2\mu_n}^\infty y^{-2-2\alpha}\,dy
 \le C_\alpha.
 \end{aligned}
\end{equation}
For every $x\ge0$, equations \eqref{eq:GH-inner-limit} and
\eqref{eq:GH-inner-bound} give
\[
 \begin{aligned}
 0\le\one_{\{x\le\mu_n\}}G_\alpha(x)(2\mu_n)^{1+2\alpha}
 \int_{2\mu_n}^\infty
 \frac{G_\alpha(y)}{(y-x)^{1+\alpha}}\,dy\le C_\alpha G_\alpha(x),
 \end{aligned}
\]
and the left-hand side converges pointwise to
$K_\alpha G_\alpha(x)/(1+2\alpha)$.  Since
$G_\alpha\in L^1(0,\infty)$, dominated convergence gives
\begin{equation}\label{eq:GH-main-range}
 \begin{aligned}
 \int_0^{\mu_n}G_\alpha(x)\int_{2\mu_n}^\infty
 \frac{G_\alpha(y)}{(y-x)^{1+\alpha}}\,dy\,dx=\frac{K_\alpha M_\alpha}{1+2\alpha}
 (2\mu_n)^{-1-2\alpha}
 +o_\alpha(\mu_n^{-1-2\alpha}).
 \end{aligned}
\end{equation}
For $\mu_n<x<2\mu_n$, put $\rho=2\mu_n-x$ and $u=y-2\mu_n$.
By \eqref{eq:G-tail-bounds-small},
\[
 \begin{aligned}
 &\int_{\mu_n}^{2\mu_n}G_\alpha(x)\int_{2\mu_n}^\infty
 \frac{G_\alpha(y)}{(y-x)^{1+\alpha}}\,dy\,dx\\
 &\quad=\int_0^{\mu_n}G_\alpha(2\mu_n-\rho)
 \int_0^\infty\frac{G_\alpha(2\mu_n+u)}
 {(\rho+u)^{1+\alpha}}\,du\,d\rho\\
 &\quad\le C_\alpha(2\mu_n)^{-2-2\alpha}
 \int_0^{\mu_n}\int_0^\infty
 (\rho+u)^{-1-\alpha}\,du\,d\rho\\
 &\quad=\frac{C_\alpha}{\alpha}(2\mu_n)^{-2-2\alpha}
 \int_0^{\mu_n}\rho^{-\alpha}\,d\rho\\
 &\quad=O_\alpha(\mu_n^{-1-3\alpha})
 =o_\alpha(\mu_n^{-1-2\alpha}).
 \end{aligned}
\]
Combining this estimate with \eqref{eq:GH-double-integral} and
\eqref{eq:GH-main-range} proves \eqref{eq:GH-expansion}.
\end{proof}
\begin{proposition}\label{prop:small}
Equation \eqref{eq:main-small} holds.
\end{proposition}

\begin{proof}
By \eqref{eq:Dexact}, \eqref{eq:FH-expansion}, and
\eqref{eq:GH-expansion},
\[
 D_n=-\frac{2c_\alpha(-1)^{n+1}K_\alpha Q_\alpha}{\alpha}
 (2\mu_n)^{-1-\alpha}
 +o_\alpha(\mu_n^{-1-\alpha}).
\]
Equations \eqref{eq:scaling} and \eqref{eq:Dexact} give
\[
 \begin{aligned}
 \frac{\calE_{\alpha,I}[w_n]}{\norm[L^2(I)]{w_n}^2}
 -\mu_n^\alpha
 =\mu_n^\alpha\frac{D_n}{\norm[L^2(0,2\mu_n)]{V_n}^2}=(-1)^n\frac{c_\alpha K_\alpha Q_\alpha}
 {\alpha2^\alpha}\mu_n^{-2}+o_\alpha(\mu_n^{-2}).
 \end{aligned}
\]
By \eqref{eq:filtered} and \eqref{eq:third}, multiplication by
$(-1)^n\mu_n^2$ followed by $n\to\infty$ gives
\[
 \frac{c_\alpha K_\alpha Q_\alpha}{\alpha2^\alpha}
 =\kappa_\alpha
 =\frac{\alpha c_\alpha}{2^{\alpha+1}}.
\]
Therefore
\begin{equation}\label{eq:Q-value}
 Q_\alpha=\frac{\alpha^2}{2K_\alpha}.
\end{equation}

Substituting \eqref{eq:FH-expansion} and \eqref{eq:GH-expansion}
into \eqref{eq:Dexact} gives
\[
 \begin{aligned}
 D_n={}&-\frac{2c_\alpha(-1)^{n+1}K_\alpha Q_\alpha}{\alpha}
 (2\mu_n)^{-1-\alpha}\\
 &+\biggl[
 -\frac{2c_\alpha(-1)^{n+1}K_{2,\alpha}Q_\alpha}{\alpha}
 +\frac{2c_\alpha(-1)^{n+1}K_\alpha\sqrt{\alpha/2}\,
 \Gamma(1-\alpha)\Gamma(1+2\alpha)}
 {\alpha\Gamma(1+\alpha)}\\
 &\qquad+\frac{2c_\alpha K_\alpha(M_\alpha-\cos\theta)}
 {1+2\alpha}\biggr](2\mu_n)^{-1-2\alpha}
 +o_\alpha(\mu_n^{-1-2\alpha}).
 \end{aligned}
\]
By \eqref{eq:Q-value}, \eqref{eq:K2small}, and
$M_\alpha-\cos\theta=-\sqrt{\alpha/2}$,
\[
 -\frac{2c_\alpha(-1)^{n+1}K_\alpha Q_\alpha}{\alpha}
 =-(-1)^{n+1}\alpha c_\alpha,
 \qquad
 \frac{2c_\alpha K_\alpha(M_\alpha-\cos\theta)}{1+2\alpha}
 =-\frac{\alpha c_\alpha^2}{1+2\alpha}.
\]
Moreover,
\[
 \begin{aligned}
 &-\frac{2c_\alpha(-1)^{n+1}K_{2,\alpha}Q_\alpha}{\alpha}
 +\frac{2c_\alpha(-1)^{n+1}K_\alpha\sqrt{\alpha/2}\,
 \Gamma(1-\alpha)\Gamma(1+2\alpha)}
 {\alpha\Gamma(1+\alpha)}\\
 &\qquad=-(-1)^{n+1}
 \frac{\alpha c_\alpha\Gamma(1+2\alpha)}{\Gamma(1+\alpha)}
 \left(2\cos\frac{\pi\alpha}{2}
 -\sec\frac{\pi\alpha}{2}\right).
 \end{aligned}
\]
Thus
\begin{equation}\label{eq:DL-full-small}
 \begin{aligned}
 D_n={}&-(-1)^{n+1}\alpha c_\alpha(2\mu_n)^{-1-\alpha}\\
 &-\left[\frac{\alpha c_\alpha^2}{1+2\alpha}
 +(-1)^{n+1}
 \frac{\alpha c_\alpha\Gamma(1+2\alpha)}{\Gamma(1+\alpha)}
 \left(2\cos\frac{\pi\alpha}{2}
 -\sec\frac{\pi\alpha}{2}\right)\right]
 (2\mu_n)^{-1-2\alpha}\\
 &+o_\alpha(\mu_n^{-1-2\alpha}).
 \end{aligned}
\end{equation}
Since \eqref{eq:Dexact} gives
$\norm[L^2(0,2\mu_n)]{V_n}^2=\mu_n+O_\alpha(1)$, we have
\[
 \frac{D_n}{\norm[L^2(0,2\mu_n)]{V_n}^2}
 =\frac{D_n}{\mu_n}+O_\alpha(\mu_n^{-3-\alpha})
 =\frac{D_n}{\mu_n}+o_\alpha(\mu_n^{-2-2\alpha}),
\]
where the last equality uses $\alpha<1$.  Equations \eqref{eq:scaling} and
\eqref{eq:DL-full-small} now give
\[
 \begin{aligned}
 \frac{\calE_{\alpha,I}[w_n]}{\norm[L^2(I)]{w_n}^2}
 ={}&\mu_n^\alpha
 +(-1)^n\frac{\alpha c_\alpha}{2^{1+\alpha}}\mu_n^{-2}-\frac{\alpha c_\alpha^2}{2^{1+2\alpha}(1+2\alpha)}
 \mu_n^{-2-\alpha}\\
 &+(-1)^n\frac{\alpha c_\alpha\Gamma(1+2\alpha)}
 {2^{1+2\alpha}\Gamma(1+\alpha)}
 \left(2\cos\frac{\pi\alpha}{2}
 -\sec\frac{\pi\alpha}{2}\right)
 \mu_n^{-2-\alpha}+o_\alpha(\mu_n^{-2-\alpha}).
 \end{aligned}
\]
By \eqref{eq:A-alpha},
\[
 -\frac{\alpha c_\alpha^2}{2^{1+2\alpha}(1+2\alpha)}
 =-\frac{2\kappa_\alpha^2}{\alpha(1+2\alpha)}=A_\alpha.
\]
Moreover,
\[
 \frac{c_\alpha}{\Gamma(1+\alpha)}
 =\frac{\sin(\pi\alpha/2)}\pi,
 \qquad
 \sin x\,(2\cos x-\sec x)=\cos(2x)\tan x.
\]
Hence
\[
 \frac{\alpha c_\alpha\Gamma(1+2\alpha)}
 {2^{1+2\alpha}\Gamma(1+\alpha)}
 \left(2\cos\frac{\pi\alpha}{2}
 -\sec\frac{\pi\alpha}{2}\right)=B_\alpha.
\]
It follows that
\[
 \frac{\calE_{\alpha,I}[w_n]}{\norm[L^2(I)]{w_n}^2}
 =\mu_n^\alpha+(-1)^n\kappa_\alpha\mu_n^{-2}
 +\bigl(A_\alpha+(-1)^nB_\alpha\bigr)\mu_n^{-2-\alpha}
 +o_\alpha(\mu_n^{-2-\alpha}).
\]
Finally, \eqref{eq:filtered} proves \eqref{eq:main-small} for the full
sequence.
\end{proof}

\begin{proof}[Proof of \Cref{thm1}]
Combine \cref{prop:small,prop:alpha-ge-one}.
\end{proof}

\section{Uniform eigenfunction bounds and nodal sets}
\label{sec:Linfty-proof}
\label{sec:nodal}

We first derive an exact representation of $\varphi_n$, then prove
\Cref{thm2,thm:nodal} in this order. Throughout this section,
$0<\alpha<2$, and all eigenfunctions are real and normalized as in
\eqref{eq:eigensystem}.

Heat-kernel domination and Plancherel, as in
\cite[proof of Proposition~2]{KwasnickiInterval}, give
\begin{equation}\label{eq:fixed-eigenfunction-bounded}
 \norm[L^\infty(I)]{\varphi_n}
 =e^{\lambda_n}\norm[L^\infty(I)]{e^{-\mathcal A_I}\varphi_n}
 \le e^{\lambda_n}
 \left(\frac1{2\pi}\int_\R e^{-2|\xi|^\alpha}\,d\xi\right)^{1/2}<\infty.
\end{equation}

\subsection{Boundary values and symmetry}
\label{subsec:nodal-preliminaries}

The eigenvalue equation gives
\begin{equation}\label{eq:nodal-weak}
 \calE_{\alpha,I}(\varphi_n,v)=\lambda_n\ip{\varphi_n}{v}
 \quad(v\in\wt H^{\alpha/2}(I)),
 \qquad
 \mathcal A\varphi_n=\lambda_n\varphi_n\quad\text{in }\mathcal D'(I).
\end{equation}

The following lemma combines boundary regularity and the Pohozaev identity
with the quadratic-form comparisons. In particular, the reflection symmetry \eqref{eq:nodal-reflection} in the case $\alpha=1$ is proved by Kulczycki–Kwaśnicki–Małecki–Stós \cite[Corollary~3]{KulczyckiKwasnickiMaleckiStos}. 

\begin{lemma}\label{lem:nodal-boundary}\label{lem:nodal-continuity}
For every $n\ge1$,
\begin{equation}\label{eq:nodal-boundary-reg}
 \begin{gathered}
 \varphi_n\in C^{\alpha/2}(\R),\qquad
 \frac{\varphi_n}{(1-|\cdot|)^{\alpha/2}}\in C(\overline I),\\
 |\varphi_n(x)|\le C_{\alpha,n}(1-|x|)^{\alpha/2},\qquad x\in I,
 \end{gathered}
\end{equation}
and
\begin{equation}\label{eq:nodal-reflection}
 \varphi_n(-x)=(-1)^{n+1}\varphi_n(x),\qquad x\in I.
\end{equation}
The boundary limits
\[
 B_+(\varphi_n)=\lim_{x\uparrow1}
       \frac{\varphi_n(x)}{(1-x)^{\alpha/2}},\qquad
 B_-(\varphi_n)=\lim_{x\downarrow-1}
       \frac{\varphi_n(x)}{(1+x)^{\alpha/2}}
\]
are nonzero.
\end{lemma}
\begin{proof}
By \eqref{eq:fixed-eigenfunction-bounded}, $\lambda_n\varphi_n\in L^\infty(I)$.
Thus \cite[Proposition~1.1 and Theorem~1.2]{RosOtonSerraRegularity}, with
$s=\alpha/2$, give \eqref{eq:nodal-boundary-reg} and the two boundary limits.

Let $Rv(x)=v(-x)$. The even and odd subspaces of $L^2(I)$ are
\[
 \begin{aligned}
 X_{\mathrm e}&=\{v\in L^2(I):v(-x)=v(x)\text{ for a.e. }x\in I\},\\
 X_{\mathrm o}&=\{v\in L^2(I):v(-x)=-v(x)\text{ for a.e. }x\in I\}.
 \end{aligned}
\]
They are orthogonal, and $v=(v+Rv)/2+(v-Rv)/2$ gives
$L^2(I)=X_{\mathrm e}\oplus X_{\mathrm o}$.
Since $R$ preserves the form domain and commutes with $\mathcal A_I$,
these subspaces are invariant. Denote the eigenvalues of the restrictions
by $\lambda_j^{\mathrm e}$ and $\lambda_j^{\mathrm o}$, respectively,
in nondecreasing order and counted with multiplicity.

Let $L_N,L_D$ be the classical Neumann and Dirichlet Laplacians on $I$,
with spectral fractional forms
\[
 Q_N[v]=\norm[L^2(I)]{L_N^{\alpha/4}v}^{2},\qquad
 Q_D[v]=\norm[L^2(I)]{L_D^{\alpha/4}v}^{2}.
\]
By \cite[Lemma~1 and Theorem~2]{MusinaNazarov} and
\cite[equation~(1) and Theorem~3]{NazarovComparison},
\[
 \begin{gathered}
 \Dom(Q_D)=\wt H^{\alpha/2}(I)
       \subset\Dom(Q_N)=H^{\alpha/2}(I),\\
 Q_N[v]\le\calE_{\alpha,I}[v]\le Q_D[v]
       \qquad(v\in\wt H^{\alpha/2}(I)).
 \end{gathered}
\]
Here zero extensions are identified with their restrictions to $I$.
The $j$th Neumann/Dirichlet frequencies are
$((j-1)\pi,(j-\tfrac12)\pi)$ in $X_{\mathrm e}$ and
$((j-\tfrac12)\pi,j\pi)$ in $X_{\mathrm o}$. Thus min--max gives
\[
 \begin{aligned}
 ((j-1)\pi)^\alpha
 &\le\lambda_j^{\mathrm e}\le((j-\tfrac12)\pi)^\alpha,\\
 ((j-\tfrac12)\pi)^\alpha
 &\le\lambda_j^{\mathrm o}\le(j\pi)^\alpha.
 \end{aligned}
\]
Applying the first line also with $j+1$, we obtain
\[
 \lambda_j^{\mathrm e}
 \le\bigl((j-\tfrac12)\pi\bigr)^\alpha
 \le\lambda_j^{\mathrm o}
 \le(j\pi)^\alpha
 \le\lambda_{j+1}^{\mathrm e},\qquad j\ge1.
\]
Simplicity in \eqref{eq:eigensystem} excludes equal even and odd eigenvalues, so
\[
 \lambda_j^{\mathrm e}<\lambda_j^{\mathrm o}<\lambda_{j+1}^{\mathrm e},
 \qquad
 \lambda_{2j-1}=\lambda_j^{\mathrm e},\quad
 \lambda_{2j}=\lambda_j^{\mathrm o}.
\]
This proves \eqref{eq:nodal-reflection}. Taking boundary limits gives
$B_-(\varphi_n)=(-1)^{n+1}B_+(\varphi_n)$.

Let $\nu(x)$ denote the outward unit normal at
$x\in\partial I=\{-1,1\}$, so $\nu(-1)=-1$, $\nu(1)=1$, and $x\nu(x)=1$.
The Pohozaev identity \cite[Theorem~1.1]{RosOtonSerraPohozaev},
with $\Omega=I$, $s=\alpha/2$, $f(t)=\lambda_nt$, and
$\norm[L^2(I)]{\varphi_n}=1$, gives
\[
 \begin{aligned}
 0<\alpha\lambda_n
 &=\Gamma(1+\alpha/2)^2
   \bigl(-\nu(-1)B_-(\varphi_n)^2+\nu(1)B_+(\varphi_n)^2\bigr)\\
 &=2\Gamma(1+\alpha/2)^2B_+(\varphi_n)^2.
 \end{aligned}
\]
Hence both boundary limits are nonzero.
\end{proof}

For the remaining fixed-eigenpair arguments, write $k_n=\lambda_n^{1/\alpha}$, and choose the sign so that
\begin{equation}\label{eq:nodal-right-sign}
 B_+(\varphi_n)>0.
\end{equation}
Constants may depend on $\alpha,n$ and on fixed test functions or compact
subintervals. By the definition of $B_+(\varphi_n)$ and
\eqref{eq:nodal-right-sign},
 for fixed $\alpha,n$, there is $0<\delta<1$ such that
\begin{equation}\label{eq:nodal-right-positive}
 \varphi_n(x)\ge\tfrac12B_+(\varphi_n)(1-x)^{\alpha/2}>0
 \qquad(1-\delta<x<1).
\end{equation}

\subsection{A strictly positive bilateral Laplace transform}
\label{subsec:nodal-laplace}

For comparison, \cite[Theorem~1.1, formula~(1.1)]{KwasnickiHalfLine}
gives a positive unilateral Laplace transform for generalized half-line
eigenfunctions. The next lemma concerns every interval eigenfunction.

\begin{lemma}\label{lem:nodal-Laplace}
Under \eqref{eq:nodal-right-sign},
\begin{equation}\label{eq:nodal-Laplace-positive}
 \mathcal L_n(t):=\int_{-1}^1e^{tx}\varphi_n(x)\,dx>0\qquad(t>0).
\end{equation}
\end{lemma}
\begin{proof}
Suppose $\mathcal L_n(t)=0$ for some $t>0$. By
\eqref{eq:nodal-reflection},
$\mathcal L_n(-t)=(-1)^{n+1}\mathcal L_n(t)=0$. Put
\begin{equation}\label{eq:nodal-test-resolvent}
 v=(-\partial_x^2+t^2)^{-1}\varphi_n,\qquad
 \widehat v(\xi)=\frac{\widehat\varphi_n(\xi)}{\xi^2+t^2}.
\end{equation}
Then $v\in H^2(\R)$ and
\[
 v(x)=\frac1{2t}\int_Ie^{-t|x-y|}\varphi_n(y)\,dy
 =\begin{cases}
 e^{-tx}\mathcal L_n(t)/(2t)=0,&x>1,\\
 e^{tx}\mathcal L_n(-t)/(2t)=0,&x<-1.
 \end{cases}
\]
Hence $v\in\wt H^{\alpha/2}(I)$. Testing \eqref{eq:nodal-weak} with
$v-\varphi_n/(k_n^2+t^2)$ gives the following contradiction:
\begin{equation}\label{eq:nodal-covariance}
 \begin{aligned}
 0&=\frac1{2\pi}\int_\R (|\xi|^\alpha-k_n^\alpha)
       \left(\frac1{\xi^2+t^2}-\frac1{k_n^2+t^2}\right)
       |\widehat\varphi_n(\xi)|^2\,d\xi\\
  &=-\frac1{2\pi(k_n^2+t^2)}\int_\R
       \frac{(|\xi|^\alpha-k_n^\alpha)(\xi^2-k_n^2)}{\xi^2+t^2}
       |\widehat\varphi_n(\xi)|^2\,d\xi<0.
 \end{aligned}
\end{equation}
Here \eqref{eq:form} gives integrability, and
$\widehat\varphi_n\ne0$ in $L^2(\R)$ gives strict positivity of the integral.
Thus $\mathcal L_n(t)\ne0$ for all $t>0$.

Using \eqref{eq:nodal-right-positive},
\[
 e^{-t(1-\delta/2)}\mathcal L_n(t)
 \ge\int_{1-\delta/2}^{1-\delta/4}\varphi_n(x)\,dx
       -e^{-t\delta/2}\norm[L^1(I)]{\varphi_n}.
\]
The integral on the right is positive, so $\mathcal L_n(t)>0$ for all
sufficiently large $t$. Since $\mathcal L_n\in C((0,\infty))$ and never
vanishes, it is positive everywhere on $(0,\infty)$.
\end{proof}

\begin{lemma}\label{lem:nodal-source}
Define $\nu_n=0$ on $I$ and
\begin{equation}\label{eq:nodal-source-integral}
 \nu_n(y)=c_\alpha\int_{-1}^1
          \frac{\varphi_n(x)}{|y-x|^{1+\alpha}}\,dx
 \qquad(|y|>1).
\end{equation}
Then $\nu_n\in L^1(\R)$ and
\begin{equation}\label{eq:nodal-source-def}
 \nu_n=\lambda_n\varphi_n-\mathcal A\varphi_n
 \quad\text{in }\mathcal D'(\R).
\end{equation}
Moreover,
\begin{equation}\label{eq:nodal-source-sign}
 \nu_n(y)>0\quad(y>1),\qquad \nu_n(-y)=(-1)^{n+1}\nu_n(y),
\end{equation}
and
\begin{equation}\label{eq:nodal-source-est}
 |\nu_n(1+r)|+|\nu_n(-1-r)|\le
 \begin{cases}
 Cr^{-\alpha/2},&0<r\le1,\\
 Cr^{-1-\alpha},&r\ge1.
 \end{cases}
\end{equation}
\end{lemma}
\begin{proof}

By
\eqref{eq:nodal-boundary-reg}, for $r>0$,
\begin{equation}\label{eq:nodal-source-near}
 \begin{aligned}
 |\nu_n(1+r)|
 &\le C\int_0^2\frac{z^{\alpha/2}}{(r+z)^{1+\alpha}}\,dz
 =Cr^{-\alpha/2}\int_0^{2/r}
       \frac{z^{\alpha/2}}{(1+z)^{1+\alpha}}\,dz\\
 &\le C\min\{r^{-\alpha/2},r^{-1-\alpha}\}.
 \end{aligned}
\end{equation}
Thus \eqref{eq:nodal-source-est} holds and $\nu_n\in L^1(\R)$.

By \eqref{eq:nodal-weak} and the separated-support formula in
\Cref{lem:operator-bound}, \eqref{eq:nodal-source-def} holds away from
$\{-1,1\}$. Fix $\eta\in C_c^\infty(\R)$ and
$\chi\in C_c^\infty((-2,2))$ with $\chi=1$ on $[-1,1]$. For $0<r<1/4$, set
\[
 \psi_r(x)=\eta(x)\left[\chi\left(\frac{x-1}{r}\right)
                    +\chi\left(\frac{x+1}{r}\right)\right].
\]
Then $\psi_r=\eta$ near $\{-1,1\}$ and, for $h>0$,
\[
 \norm[L^1(\R)]{\psi_r(\cdot+h)-\psi_r}
 \le\min\{h\norm[L^1(\R)]{\psi_r'},2\norm[L^1(\R)]{\psi_r}\}
 \le C_\eta\min\{h,r\}.
\]
The bilinear form of \eqref{eq:translation-form} and
$\varphi_n\in C^{\alpha/2}(\R)$ give
\begin{equation}\label{eq:nodal-no-delta-test}
 \begin{aligned}
 |\langle\mathcal A\varphi_n,\psi_r\rangle|
 &\le c_\alpha\int_0^\infty
 \frac{\norm[L^\infty(\R)]{\varphi_n(\cdot+h)-\varphi_n}
       \norm[L^1(\R)]{\psi_r(\cdot+h)-\psi_r}}{h^{1+\alpha}}\,dh\\
 &\le C\int_0^\infty\frac{\min\{h,r\}}{h^{1+\alpha/2}}\,dh
 \le Cr^{1-\alpha/2}\longrightarrow0.
 \end{aligned}
\end{equation}
Since $\lambda_n\varphi_n-\nu_n\in L^1(\R)$, dominated convergence gives
\[
 \begin{aligned}
 \langle\lambda_n\varphi_n-\mathcal A\varphi_n-\nu_n,\eta\rangle
 &=\langle\lambda_n\varphi_n-\nu_n,\psi_r\rangle
   -\langle\mathcal A\varphi_n,\psi_r\rangle
 \longrightarrow0.
 \end{aligned}
\]
This proves \eqref{eq:nodal-source-def} on all of $\R$.

For $y>1$, the Gamma integral, Fubini, and
\Cref{lem:nodal-Laplace} yield
\begin{equation}\label{eq:nodal-source-Laplace}
	\nu_n(y)=\frac{c_\alpha}{\Gamma(1+\alpha)}
	\int_0^\infty t^\alpha e^{-ty}\mathcal L_n(t)\,dt>0.
\end{equation}
Fubini is justified by
$\int_I|\varphi_n(x)|(y-x)^{-1-\alpha}\,dx<\infty$. The identity $\nu_n(-y)=(-1)^{n+1}\nu_n(y)$ follows from
\eqref{eq:nodal-reflection}. This proves \eqref{eq:nodal-source-sign}.
\end{proof}

Define
\begin{equation}\label{eq:nodal-b-def}
 \beta_n(t)=\int_1^\infty e^{-ty}\nu_n(y)\,dy\qquad(t>0).
\end{equation}
By \eqref{eq:nodal-source-est},
\[
 0<\beta_n(t)\le\norm[L^1(\R)]{\nu_n},\qquad
 \beta_n(t)\le Ce^{-t}\int_0^\infty e^{-tr}r^{-\alpha/2}\,dr
 =C\Gamma(1-\alpha/2)e^{-t}t^{\alpha/2-1}.
\]
Thus
\begin{equation}\label{eq:nodal-b-est}
 0<\beta_n(t)\le
 \begin{cases}
 C,&0<t\le1,\\
 Ce^{-t}t^{\alpha/2-1},&t\ge1.
 \end{cases}
\end{equation}

\subsection{An explicit representation for the eigenfunction}
\label{subsec:nodal-representation}

The following formula specializes
\cite[Proposition~2.17(b), Corollary~2.16(b), and Proposition~2.18]{KwasnickiHalfLine}
to $f(z)=z^{\alpha/2}$. The same quotient appears in
\cite[Proposition~3.1]{DeNittiFernandezReal},
with $s=\alpha/2$ and $\theta=2/\alpha$.

\begin{lemma}\label{lem:nodal-multiplier}
For $a=\lambda_n^{2/\alpha}>0$, define
\begin{equation}\label{eq:nodal-q-def}
 Q_n(z)=\frac{z-a}{z^{\alpha/2}-a^{\alpha/2}}\quad(z\ge0,\ z\ne a),
 \qquad Q_n(a)=\frac{2a^{1-\alpha/2}}{\alpha}.
\end{equation}
Then
\begin{equation}\label{eq:nodal-q-rep}
 Q_n(z)=a^{1-\alpha/2}+\int_0^\infty\frac{z}{z+r}\rho_n(r)\,dr,
\end{equation}
where
\begin{equation}\label{eq:nodal-rho}
 \rho_n(r)=\frac{\sin(\pi\alpha/2)}\pi
 \frac{(r+a)r^{\alpha/2-1}}
      {r^{\alpha}-2\lambda_n r^{\alpha/2}\cos(\pi\alpha/2)+\lambda_n^2}>0.
\end{equation}
Moreover,
\begin{equation}\label{eq:nodal-rho-integrability}
 \begin{gathered}
 \rho_n(r)=O(r^{\alpha/2-1})\quad(r\downarrow0),\qquad
 \rho_n(r)=O(r^{-\alpha/2})\quad(r\to\infty),\\
 \int_0^\infty\frac{\rho_n(r)}{1+r}\,dr<\infty.
 \end{gathered}
\end{equation}
\end{lemma}
\begin{proof}
Since $Q_n(0)=a^{1-\alpha/2}$ and $\lim_{z\to\infty}Q_n(z)/z=0$,
\cite[Proposition~2.17(b) and Corollary~2.16(b)]{KwasnickiHalfLine} give
\[
 Q_n(z)=a^{1-\alpha/2}+\int_0^\infty\frac{z}{z+r}\,\sigma(dr),
 \qquad \int_0^\infty\frac{\sigma(dr)}{1+r}<\infty,
\]
where $\sigma$ is a nonnegative Borel measure on $(0,\infty)$.
Using the principal branch of $z^{\alpha/2}$, the boundary limits
$Q_n(-r+i0)$ are locally uniform for $r>0$.
Indeed, $(-r+i\varepsilon)^{\alpha/2}\to r^{\alpha/2}e^{i\pi\alpha/2}$
uniformly for $r\in[r_0,r_1]\subset(0,\infty)$ as $\varepsilon\downarrow0$,
and the limiting denominator satisfies
\[
 |r^{\alpha/2}e^{i\pi\alpha/2}-\lambda_n|
 \ge r_0^{\alpha/2}\sin(\pi\alpha/2)>0.
\]
Stieltjes inversion
\cite[Proposition~2.18]{KwasnickiHalfLine}, applied to
$(Q_n(z)-Q_n(0))/z$, therefore gives
\[
 \begin{aligned}
 \sigma(dr)&=\frac{\operatorname{Im}Q_n(-r+i0)}{\pi r}\,dr\\
 &=\frac{\sin(\pi\alpha/2)}\pi
   \frac{(r+a)r^{\alpha/2-1}}
        {|r^{\alpha/2}e^{i\pi\alpha/2}-\lambda_n|^2}\,dr
   =\rho_n(r)\,dr.
 \end{aligned}
\]
This proves \eqref{eq:nodal-q-rep}. Finally,
\begin{equation}\label{eq:nodal-denominator-bound}
 |r^{\alpha/2}e^{i\pi\alpha/2}-\lambda_n|^2
 \ge(1-|\cos(\pi\alpha/2)|)(r^\alpha+\lambda_n^2)>0
 \qquad(r>0),
\end{equation}
so \eqref{eq:nodal-rho} gives \eqref{eq:nodal-rho-integrability}.
\end{proof}

Now we can establish an explicit formula for the eigenfunction $\varphi_n$.
\begin{proposition}\label{prop:nodal-representation}
Define
\begin{align}
 \mathcal K_n(t)&=\frac{\sin(\pi\alpha/2)}\pi
 \frac{t^\alpha\beta_n(t)}
      {t^{2\alpha}-2\lambda_nt^\alpha\cos(\pi\alpha/2)+\lambda_n^2}>0,
 \label{eq:nodal-K}\\
 \mathcal H_n(x)&=\int_0^\infty\mathcal K_n(t)
       (e^{tx}+(-1)^{n+1}e^{-tx})\,dt,\qquad -1\le x\le1.
 \label{eq:nodal-h}
\end{align}
Then $\varphi_n,\mathcal H_n\in C(\bar{I})\cap C^\infty(I)$ and, for some
$c_n\in\R\setminus\{0\}$,
\begin{equation}\label{eq:nodal-exact-representation}
 \varphi_n(x)=c_nT_n(x)-\mathcal H_n(x),\qquad
 T_n(x)=\begin{cases}
       \cos(k_nx),&n\text{ odd},\\
       \sin(k_nx),&n\text{ even}.
       \end{cases}
\end{equation}
Moreover,
\begin{gather}
 \mathcal H_n^{(j)}(x)>0\qquad(j=0,1,2,\quad0<x<1),
 \label{eq:nodal-h-signs}\\
 c_nT_n(1)=\mathcal H_n(1)>0,\qquad T_n(1)\ne0,
 \label{eq:nodal-nonresonance}\\
 0\le\mathcal H_n(x)<\mathcal H_n(1)\le|c_n|
       \qquad(0\le x<1).
 \label{eq:nodal-amplitude}
\end{gather}
For odd $n$, $\mathcal H_n(0)>0$ and $\mathcal H_n'(0)=0$.
For even $n$, $\mathcal H_n(0)=0$ and $\mathcal H_n'(0)>0$.
\end{proposition}
\begin{proof}
Equations~\eqref{eq:nodal-b-est} and \eqref{eq:nodal-denominator-bound} give
\begin{equation}\label{eq:nodal-K-est}
 0<\mathcal K_n(t)\le
 \begin{cases}
 Ct^\alpha,&0<t\le1,\\
 Ce^{-t}t^{-1-\alpha/2},&t\ge1.
 \end{cases}
\end{equation}
Thus, for every integer $j\ge0$ and $0<\delta<1$,
\[
 \int_0^\infty\bigl(e^t+t^je^{(1-\delta)t}\bigr)\mathcal K_n(t)\,dt<\infty.
\]
Hence $\mathcal H_n\in C(\bar{I})\cap C^\infty(I)$ and
\begin{equation}\label{eq:nodal-h-derivatives}
 \mathcal H_n^{(j)}(x)=\int_0^\infty t^j\mathcal K_n(t)
       (e^{tx}+(-1)^{n+j+1}e^{-tx})\,dt.
\end{equation}
For $x>0$, $e^{tx}\pm e^{-tx}>0$, while at $x=0$ the factor in parentheses
is $1+(-1)^{n+j+1}$. This gives \eqref{eq:nodal-h-signs} and the assertions at $0$.

By \Cref{lem:nodal-source,lem:nodal-multiplier},
\begin{equation}\label{eq:nodal-mu-Fourier}
 \widehat\nu_n=(\lambda_n-|\xi|^\alpha)\widehat\varphi_n,\qquad
 Q_n(\xi^2)\widehat\nu_n(\xi)=(k_n^2-\xi^2)\widehat\varphi_n(\xi).
\end{equation}
For $\psi\in C_c^\infty(I)$, Fubini is justified by
\[
 \begin{aligned}
 &\int_0^\infty\rho_n(r)\int_\R\frac{\xi^2}{\xi^2+r}
       |\widehat\nu_n(\xi)\widehat\psi(-\xi)|\,d\xi\,dr\\
 &\qquad\le C\norm[L^1(\R)]{\nu_n}
       \int_\R(1+\xi^2)^{1-\alpha/2}|\widehat\psi(-\xi)|\,d\xi<\infty,
 \end{aligned}
\]
since \eqref{eq:nodal-q-def} gives $Q_n(z)\le C(1+z)^{1-\alpha/2}$.
Using \eqref{eq:nodal-q-rep}, $\xi^2/(\xi^2+r)=1-r/(\xi^2+r)$,
and $\nu_n=0$ on $I$, we obtain
\begin{equation}\label{eq:nodal-source-pairing}
 \begin{aligned}
 \langle(\partial_x^2+k_n^2)\varphi_n,\psi\rangle
 &=\frac1{2\pi}\int_\R Q_n(\xi^2)\widehat\nu_n(\xi)
                      \widehat\psi(-\xi)\,d\xi\\
 &=-\int_0^\infty r\rho_n(r)
       \langle(-\partial_x^2+r)^{-1}\nu_n,\psi\rangle\,dr.
 \end{aligned}
\end{equation}
For $x\in I$, the resolvent kernel and \eqref{eq:nodal-source-sign} give
\begin{equation}\label{eq:nodal-source-resolvent}
 \begin{aligned}
 (-\partial_x^2+t^2)^{-1}\nu_n(x)
 &=\frac1{2t}\int_\R e^{-t|x-y|}\nu_n(y)\,dy\\
 &=\frac{\beta_n(t)}{2t}(e^{tx}+(-1)^{n+1}e^{-tx}).
 \end{aligned}
\end{equation}
Since $t^2\rho_n(t^2)\beta_n(t)=(t^2+k_n^2)\mathcal K_n(t)$,
the substitution $r=t^2$ in \eqref{eq:nodal-source-pairing} yields, in $\mathcal D'(I)$,
\[
 \begin{aligned}
 (\partial_x^2+k_n^2)\varphi_n
 &=-\int_0^\infty(t^2+k_n^2)\mathcal K_n(t)
        (e^{tx}+(-1)^{n+1}e^{-tx})\,dt\\
 &=-(\partial_x^2+k_n^2)\mathcal H_n.
 \end{aligned}
\]
The integral converges locally with all derivatives by \eqref{eq:nodal-K-est}.
Thus
\begin{equation}\label{eq:nodal-ode}
 (\partial_x^2+k_n^2)(\varphi_n+\mathcal H_n)=0.
\end{equation}
Its distributional solutions lie in
$\operatorname{span}\{\cos(k_nx),\sin(k_nx)\}$.
Since $\mathcal{H}_n(-x)=(-1)^{n+1}\mathcal{H}_n(x)$ and \eqref{eq:nodal-reflection},  we have \eqref{eq:nodal-exact-representation}
and $\varphi_n\in C^\infty(I)$.
Finally, continuity and $\varphi_n(1)=0$ give
$c_nT_n(1)=\mathcal H_n(1)>0$, hence $c_n\ne0$ and $T_n(1)\ne0$.
Since $|T_n(1)|\le1$, strict monotonicity of $\mathcal H_n$
gives \eqref{eq:nodal-amplitude}.
\end{proof}

\subsection{Proof of \Cref{thm2}}
\label{subsec:eigenfunction-bounds}

Choose the sign \eqref{eq:nodal-right-sign} and use
\Cref{prop:nodal-representation}. Since $T_n(1)\ne0$ excludes
$k_n=n\pi/2$, \eqref{eq:sandwich} gives
\begin{equation}\label{eq:nodal-frequency-comparison}
 \frac{(n-1)\pi}{2}\le k_n<\frac{n\pi}{2},\qquad k_n>0.
\end{equation}
Set
\[
 A=|c_n|,\qquad
 M=\norm[L^\infty(I)]{\varphi_n}=\max_{[0,1]}|\varphi_n|,
 \qquad \int_0^1\varphi_n^2\,dx=\frac12,
\]
where the last two equalities follow from reflection symmetry.

\medskip\noindent\textbf{The cases $n=1,2$.}
By \eqref{eq:nodal-frequency-comparison}, $T_n>0$ on $(0,1]$,
so \eqref{eq:nodal-nonresonance} gives $c_n>0$. Hence
\[
 \varphi_n''=-k_n^2c_nT_n-\mathcal H_n''<0\quad\text{on }(0,1),
 \qquad \varphi_n(0)\ge0,\qquad\varphi_n(1)=0.
\]
Thus $\varphi_n\ge0$ on $[0,1]$. If $\varphi_n(p)=M$ with
$p\in[0,1)$, concavity gives
\[
 \varphi_n(x)\ge
 \begin{cases}
 Mx/p,&0\le x\le p,\quad p>0,\\
 M(1-x)/(1-p),&p\le x\le1.
 \end{cases}
\]
Squaring and integrating yields
$\tfrac12\ge M^2\bigl(p/3+(1-p)/3\bigr)=M^2/3$, and therefore
\begin{equation}\label{eq:first-two-Linfty}
 \norm[L^\infty(I)]{\varphi_n}\le\sqrt{3/2},\qquad n=1,2.
\end{equation}

\medskip\noindent\textbf{The cases $n=3,4$.}
Now $T_n(1)<0$ by \eqref{eq:nodal-frequency-comparison}, so $c_n=-A$.
Put
\[
 v=-\varphi_n=AT_n+\mathcal H_n,\qquad
 \ell=\frac{\pi}{2k_n}\le\frac1{n-1}.
\]
On $((n-2)\ell,1)$, $T_n<0$ and $v''>0$.
Since $v((n-2)\ell)=\mathcal H_n((n-2)\ell)<A$ and $v(1)=0$,
convexity gives $v\le A$ on $[(n-2)\ell,1]$.
For $n=3$, $v(0)=A+\mathcal H_n(0)>A$.
For $n=4$, $v(\ell)=A+\mathcal H_n(\ell)>A$ and
\[
 v'(x)=Ak_n\cos(k_nx)+\mathcal H_n'(x)>0
 \qquad(0<x\le\ell).
\]
As $\varphi_n=-AT_n-\mathcal H_n\le A$ on $[0,1]$, we can choose
\[
 M=v(p)>A,\qquad
 p\in\begin{cases}[0,\ell),&n=3,\\(\ell,2\ell),&n=4,\end{cases}
 \qquad v'(p)=0.
\]
If $n=3$ and $p=0$, the derivative vanishes by evenness.
Since $\ell\le1/(n-1)$,
\[
 [p-\ell,p+\ell]\subset
 \begin{cases}I,&n=3,\\(0,1),&n=4.\end{cases}
\]
On this interval,
$v''+k_n^2v=\mathcal H_n''+k_n^2\mathcal H_n\ge0$,
using evenness of $\mathcal H_n$ when $n=3$.
For $0\le t\le\ell$, variation of constants gives
\[
 \begin{aligned}
 v(p\pm t)
 &=M\cos(k_nt)+\frac1{k_n}\int_0^t\sin(k_n(t-s))
   (\mathcal H_n''+k_n^2\mathcal H_n)(p\pm s)\,ds\\
 &\ge M\cos(k_nt)\ge0,
 \end{aligned}
\]
since $0\le k_n(t-s)\le\pi/2$ for $0\le s\le t\le\ell$.
Consequently,
\[
 \frac{\pi M^2}{2k_n}
 =2M^2\int_0^\ell\cos^2(k_nt)\,dt
 \le\int_{p-\ell}^{p+\ell}\varphi_n^2\,dx
 \le\begin{cases}1,&n=3,\\1/2,&n=4.\end{cases}
\]
Together with \eqref{eq:nodal-frequency-comparison}, this proves
\begin{equation}\label{eq:third-fourth-Linfty}
 \norm[L^\infty(I)]{\varphi_3}\le\sqrt3,\qquad
 \norm[L^\infty(I)]{\varphi_4}\le\sqrt2.
\end{equation}

\medskip\noindent\textbf{The case $n\ge5$.}
Set
\[
 r=\frac{(n-2)\pi}{2k_n},\qquad h=\mathcal H_n(r/2).
\]
By \eqref{eq:nodal-frequency-comparison},
\[
 \frac35\le\frac{n-2}{n}<r\le\frac{n-2}{n-1}<1.
\]
The point $r$ is a zero of $T_n$, and the next zero is
$n\pi/(2k_n)>1$. Thus \eqref{eq:nodal-nonresonance} gives
\[
 c_nT_n>0\quad\text{on }(r,1),\qquad c_nT_n'(r)=Ak_n.
\]
Since $T_n''=-k_n^2T_n$, $T_n'(0)\mathcal H_n(0)=0$,
$|c_nT_n'|\le Ak_n$, and $\mathcal H_n'\ge0$ on $[0,r]$,
integration by parts gives
\[
 \begin{aligned}
 k_n^2\int_0^r c_nT_n\mathcal H_n\,dx
 &=-[c_nT_n'\mathcal H_n]_0^r
   +\int_0^r c_nT_n'\mathcal H_n'\,dx\\
 &\le-Ak_n\mathcal H_n(r)
      +Ak_n\bigl(\mathcal H_n(r)-\mathcal H_n(0)\bigr)\\
 &=-Ak_n\mathcal H_n(0)\le0.
 \end{aligned}
\]
Also, $\int_0^rT_n^2\,dx=r/2$ and
$\mathcal H_n(x)\ge h$ for $r/2\le x\le r$. Expanding the square yields
\begin{equation}\label{eq:amplitude-energy}
 \begin{aligned}
 1\ge2\int_0^r\varphi_n^2\,dx
 &\ge A^2r+2\int_0^r\mathcal H_n^2\,dx
 \ge r(A^2+h^2).
 \end{aligned}
\end{equation}
On $[0,r]$, $|\varphi_n|\le A+\mathcal H_n(r)$.
On $[r,1]$, $\varphi_n$ is concave with
$\varphi_n(r)=-\mathcal H_n(r)$ and $\varphi_n(1)=0$, so
\[
 -\mathcal H_n(r)\le
 -\frac{1-x}{1-r}\mathcal H_n(r)\le\varphi_n(x)
 \le c_nT_n(x)\le A.
\]
Hence
\begin{equation}\label{eq:maximum-amplitude}
 M\le A+\mathcal H_n(r).
\end{equation}
Convexity of $\mathcal H_n$ on $[r/2,1]$ and
$\mathcal H_n(1)\le A$ give
\[
 \mathcal H_n(r)\le\frac{2(1-r)h+rA}{2-r},\qquad
 M\le\frac{2}{2-r}\bigl(A+(1-r)h\bigr).
\]
Cauchy--Schwarz and \eqref{eq:amplitude-energy} now imply
\begin{equation}\label{eq:uniform-two-bound}
 M^2\le\frac{4(1+(1-r)^2)(A^2+h^2)}{(2-r)^2}
 \le\frac{4(1+(1-r)^2)}{r(2-r)^2}\le4,
\end{equation}
since
\[
 r(2-r)^2-\bigl(1+(1-r)^2\bigr)
 =(1-r)(4r-r^2-2)\ge0\qquad(3/5\le r\le1).
\]
Together with \eqref{eq:first-two-Linfty} and \eqref{eq:third-fourth-Linfty},
this proves the bound $M\le2$ for all $n\ge1$ and $0<\alpha<2$.

\subsection{Proof of \Cref{thm:nodal}}
\label{subsec:nodal-counting}

For a bounded interval $J$, let $Z_J(f)=\{x\in J:f(x)=0\}$. The proof of \Cref{thm:nodal} uses the exact eigenfunction representation and a concavity argument to show that \(\varphi_n\) and \(T_n\) have the same number of zeros in \((0,1)\), with every zero of \(\varphi_n\) being simple. The eigenvalue bounds and the reflection symmetry of $\varphi_n$, together with a separate analysis at \(0\), then yield exactly \(n-1\) simple zeros in \(I\).
\begin{proposition}\label{prop:nodal-zero-count}
With the sign \eqref{eq:nodal-right-sign} and the functions in
\Cref{prop:nodal-representation},
\begin{equation}\label{eq:nodal-positive-half-count}
 \#Z_{(0,1)}(\varphi_n)=\#Z_{(0,1)}(T_n).
\end{equation}
Every zero of $\varphi_n$ is simple. At $0$, $\varphi_n(0)\ne0$ for odd $n$,
whereas $\varphi_n(0)=0$ and $\varphi_n'(0)\ne0$ for even $n$.
\end{proposition}
\begin{proof}
By \Cref{prop:nodal-representation},
\begin{equation}\label{eq:nodal-negative-components}
 c_nT_n(x)\le0\quad\Longrightarrow\quad
 \varphi_n(x)=c_nT_n(x)-\mathcal H_n(x)<0\qquad(0<x<1).
\end{equation}
Let $J=(\ell,r)$ be a connected component of
$\{x\in(0,1):c_nT_n(x)>0\}$. On $J$,
\begin{equation}\label{eq:nodal-strict-concavity}
 \varphi_n''=-k_n^2c_nT_n-\mathcal H_n''<0.
\end{equation}
If $r<1$, then $c_nT_n$ attains $|c_n|$ on $[\ell,r)$.
Since $\mathcal H_n<|c_n|$ there by \eqref{eq:nodal-amplitude},
continuity gives $p\in J$ with $\varphi_n(p)>0$.
If $r=1$, such a point exists by \eqref{eq:nodal-right-positive}.

If $\ell>0$, then $T_n(\ell)=0$ and
$\varphi_n(\ell)=-\mathcal H_n(\ell)<0$.
If $\ell=0$, then $c_n>0$ and $\varphi_n(0)\ge0$
by \eqref{eq:nodal-amplitude}.
Similarly, $\varphi_n(r)=-\mathcal H_n(r)<0$ if $r<1$,
whereas $\varphi_n(r)=0$ if $r=1$.

For $e\in\{\ell,r\}$ with $\varphi_n(e)\ge0$, concavity gives
\[
 \varphi_n((1-t)p+te)
 \ge(1-t)\varphi_n(p)+t\varphi_n(e)>0,
 \qquad 0<t<1.
\]
Thus there are no zeros strictly between $p$ and $e$.
If $\varphi_n(e)<0$, continuity gives at least one zero between
$p$ and $e$. If two existed, label them $z_1,z_2$ so that
$z_2=(1-t)z_1+tp$ for some $t\in(0,1)$. Concavity would give
\[
 0=\varphi_n(z_2)
 \ge(1-t)\varphi_n(z_1)+t\varphi_n(p)
 =t\varphi_n(p)>0,
\]
a contradiction. Consequently,
\[
 \#Z_J(\varphi_n)=\one_{\{\ell>0\}}+\one_{\{r<1\}}.
\]
At every zero $z\in J$, the tangent-line inequality for concave functions gives
\begin{equation}\label{eq:nodal-derivative-sign}
 0<\varphi_n(p)
 \le\varphi_n(z)+\varphi_n'(z)(p-z)
 =\varphi_n'(z)(p-z),
\end{equation}
so $\varphi_n'(z)\ne0$.

Each zero of $T_n$ in $(0,1)$ is simple, so $c_nT_n$ changes sign
there and the zero is an endpoint of exactly one such $J$.
Summing over these finitely many intervals and using
\eqref{eq:nodal-negative-components}, we obtain
\begin{equation}\label{eq:nodal-count-sum}
 \#Z_{(0,1)}(\varphi_n)
 =\sum_{J=(\ell,r)}
   \bigl(\one_{\{\ell>0\}}+\one_{\{r<1\}}\bigr)
 =\#Z_{(0,1)}(T_n).
\end{equation}

For odd $n$, $0<\mathcal H_n(0)<|c_n|$ gives
$\varphi_n(0)=c_n-\mathcal H_n(0)\ne0$.
For even $n$, $\varphi_n(0)=0$.
If $c_n<0$, then
$\varphi_n'(0)=c_nk_n-\mathcal H_n'(0)<0$.
If $c_n>0$, the component with $\ell=0$ contains a point $p>0$
with $\varphi_n(p)>0$, as above. Since $\varphi_n''<0$ on $(0,p)$,
\[
 0<\varphi_n(p)=\int_0^p\varphi_n'(x)\,dx
 \le p\varphi_n'(0).
\]
Thus $\varphi_n'(0)\ne0$.
\end{proof}

By \Cref{prop:nodal-representation}, $\varphi_n\in C^\infty(I)$.
Choose the sign \eqref{eq:nodal-right-sign}.
Equation~\eqref{eq:nodal-frequency-comparison} gives
\[
 m\pi\le k_{2m+1}<(m+\tfrac12)\pi\quad(m\ge0),\qquad
 (m-\tfrac12)\pi\le k_{2m}<m\pi\quad(m\ge1).
\]
The trigonometric zeros, \Cref{prop:nodal-zero-count}, and the reflection symmetry
\eqref{eq:nodal-reflection} give
\begin{align*}
 \#Z_I(\varphi_{2m+1})
 &=2\#\{j\in\mathbb Z_{\ge0}:(j+\tfrac12)\pi<k_{2m+1}\}
 =2m,\\
 \#Z_I(\varphi_{2m})
 &=2\#\{j\in\mathbb Z_{\ge1}:j\pi<k_{2m}\}+1
 =2m-1.
\end{align*}
The central value is nonzero in the first line and is a simple zero in
the second. Simplicity on $(0,1)$ follows from \Cref{prop:nodal-zero-count};
on $(-1,0)$ it follows from
\[
 \varphi_n'(-x)=(-1)^n\varphi_n'(x).
\]
Finally, at every zero $z$,
\[
 \varphi_n(z+h)=h\varphi_n'(z)+o(h),\qquad\varphi_n'(z)\ne0.
\]
Thus $\varphi_n$ changes sign at each of its $n-1$ zeros, and
$I\setminus Z_I(\varphi_n)$ has exactly $n$ connected components.
So we complete the proof of \Cref{thm:nodal}.

\end{document}